\documentclass[12pt]{amsart}
\usepackage{fullpage}
\usepackage{graphicx}
\usepackage{amssymb}
\usepackage[font={footnotesize,it}]{caption}
\usepackage{comment}
\usepackage{url}
\newcommand{\Z}{\mathbb{Z}}
\newcommand{\Q}{\mathbb{Q}}
\newcommand{\R}{\mathbb{R}}
\newcommand{\C}{\mathbb{C}}
\newcommand{\A}{\mathbb{A}}
\newtheorem{theorem}{Theorem}
\newtheorem{proposition}[theorem]{Proposition}
\newtheorem{lemma}[theorem]{Lemma}
\newtheorem{conjecture}[theorem]{Conjecture}
\newtheorem{definition}[theorem]{Definition}

\numberwithin{theorem}{section}
\theoremstyle{remark}
\newtheorem*{remark}{Remark}
\DeclareMathOperator{\sinc}{sinc}
\DeclareMathOperator{\ord}{ord}
\DeclareMathOperator{\sgn}{sgn}
\DeclareMathOperator{\eigen}{eigen}
\begin{document}
\title{Detecting squarefree numbers}
\author{Andrew R.\ Booker\and Ghaith A.\ Hiary\and Jon P.\ Keating}
\thanks{A.~R.~B.\ was supported by EPSRC Fellowship EP/H005188/1.
J.~P.~K.\ was sponsored by the Air Force Office of Scientific Research,
Air Force Material Command, USAF, under grant number FA8655-10-1-3088.
The US Government is authorized to reproduce and distribute reprints for
Governmental purpose notwithstanding any copyright notation thereon.
J.~P.~K. and G.~A.~H.\ also gratefully acknowledge support from the
Leverhulme Trust.}
\address{A.~R.~B. and J.~P.~K.: School of Mathematics, University of Bristol, University Walk,
Bristol, BS8 1TW, United Kingdom}
\address{G.~A.~H.: Department of Mathematics, The Ohio State University, 231 West 18th Ave, Columbus, OH 43210,
United States}
\begin{abstract}
We present an algorithm, based on the explicit formula for $L$-functions
and conditional on GRH, for proving that a given integer is squarefree
with little or no knowledge of its factorization.  We analyze the
algorithm both theoretically and practically, and use it to prove that
several RSA challenge numbers are not squarefull.
\end{abstract}
\maketitle
\section{Introduction}
Let $k$ be a finite field and $f$ a non-zero element of $k[x]$.  Then it
is well known that $f$ is squarefree if and only if $\gcd(f,f')=1$,
and the latter condition may be checked quickly (in deterministic
polynomial time) by the Euclidean algorithm.  It is a long-standing
question in algorithmic number theory whether there is a correspondingly
simple procedure to test if a given integer is squarefree; in particular,
can one determine whether $N\in\Z$ is squarefree more rapidly than by
factoring it?

In this paper, we describe an algorithm, conditional on the Generalized
Riemann Hypothesis (GRH), for proving an integer squarefree with little
or no knowledge of its factorization, and analyze the complexity of
the algorithm both theoretically and practically.  In particular,
we present some heuristic evidence based on random matrix theory and other probabilistic calculations
that our algorithm runs in deterministic subexponential time
$O\bigl(\exp[(\log{N})^{2/3+o(1)}]\bigr)$. Although this is poorer than
the performance expected of the current best known factoring algorithms,
our method is able to give partial results that one does not obtain
from a failed attempt at factoring.  In particular, we show the following
(see \S\ref{numerics}).
\begin{theorem}
\label{numericalresult}
Assume GRH for quadratic Dirichlet
$L$-functions.  Then the RSA challenge numbers {\tt RSA-210}, {\tt
RSA-220}, {\tt RSA-230} and {\tt RSA-232} are not
squarefull, i.e.\ each has at least one prime factor of multiplicity $1$.
\end{theorem}
The challenge numbers mentioned in the theorem, ranging in size from
210 to 232 digits, are significant because they are the smallest
that have yet to be factored.\footnote{{\tt RSA-210} was factored in
September 2013, after this paper was submitted but before publication.}
Certainly the technology to factor them exists (in fact the comparably
sized\footnote{{\tt RSA-704} and {\tt RSA-768} are named for their sizes
in binary; in decimal they have 212 and 232 digits, respectively.}
{\tt RSA-704} and {\tt RSA-768} were successfully factored in 2012
and 2009, respectively), but it remains prohibitively expensive to
perform such factorizations routinely.  In contrast, the proof of
Theorem~\ref{numericalresult} for {\tt RSA-210}
could be carried out with a desktop PC
in a few months. To our knowledge, Theorem~\ref{numericalresult} is the
first statement of its kind to be proven without exhibiting any factors
of the number in question.

\subsection*{Acknowledgements}
We thank Paul Bourgade, Peter Sarnak and Akshay Venkatesh for helpful
conversations.

\subsection{Background}
We begin with some background on the problem of squarefree testing,
before describing our main algorithm in \S\ref{explicit}.
Given an integer $N>1$, we first note that if $N$ has no prime factors
$\le\sqrt[3]{N}$ then it is squarefree if and only if it is not a
perfect square.  Thus, since it is easy to detect squares, in order to
prove a number squarefree it suffices to find all of its prime factors
up to the cube root.  On the other hand, the Pollard--Strassen algorithm
\cite{pollard, strassen} finds all prime factors of $N$ up to a given
bound $B$ in time $O_\varepsilon\bigl(N^\varepsilon\sqrt{B}\bigr)$.
This immediately yields an algorithm for squarefree testing in
time $O_\varepsilon\bigl(N^{1/6+\varepsilon}\bigr)$.  We remark
that with some modifications to the Pollard--Strassen algorithm,
along the lines of \cite{costa-harvey} but specific to this
problem, one can improve the running time very slightly to
$O\bigl(N^{\frac16-\frac{c}{\log\log{N}}}\bigr)$ for some $c>0$.

Although Pollard--Strassen is often regarded as a purely theoretical
result, with modern computers it is possible to implement it and
realize some improvement in speed over trial division.  However, the
gains do not occur until $B$ is of size $10^9$ at least.  As a result,
even the modified algorithm mentioned above is only practical for $N$
up to $10^{70}$ or so.  On the other hand, the Quadratic Sieve algorithm
running on a PC will, in practice, almost surely factor a given $N\le
10^{70}$ within a few minutes; thus, at least with present algorithms
and technology, it is always better to try to factor the given integer.

\subsection{Fundamental discriminants}
\label{fdsection}
Our approach rests on a way of characterizing the squarefree integers
that does not directly refer to their factorization.  Precisely,
if $d\in\Z$, $d\equiv1\pmod4$, then $d$ is squarefree if and only
if it is a fundamental discriminant.  (Note that if $N\in\Z$
is odd then $d=(-1)^\frac{N-1}2N$ satisfies $d\equiv1\pmod4$,
so this restriction entails no loss of generality.) The
advantage of this criterion is that whether or not a given
discriminant $d$ is fundamental can be detected from values of
the quadratic character $\chi_d(n)=\left(\frac{d}{n}\right)$, where
$\left(\frac{\;\;}{\;\;}\right)$ denotes the Kronecker symbol.  In turn,
$\chi_d(n)$ is easy to compute for a given $n$, thanks to quadratic
reciprocity; in particular, if $n$ is a prime then the Kronecker symbol
$\left(\frac{d}{n}\right)$ reduces to the Legendre symbol, which can be
evaluated, e.g., by Euler's criterion.

Let $\mathcal{F}$ denote the set of fundamental discriminants.
To see how one might use the above to prove quickly that a given $d$ is
squarefree, note first that we have in general that $d=\Delta\ell^2$,
where $\Delta\in\mathcal{F}$ and $\ell\in\Z_{>0}$. Here
$|\Delta|$ is an invariant of the character $\chi_d$ (its conductor),
which we aim to show equals, or is as least close to, $|d|$.
By testing whether $d$ is a square, we may assume without
loss of generality that $\Delta\ne1$.

For any $x>0$, consider the series
\begin{equation}
\label{thetaseries}
S_\Delta(x)=\frac1{\sqrt{x}}\sum_{n=1}^{\infty}
\chi_\Delta(n)\left(\frac{n}{x}\right)^{(1-\chi_\Delta(-1))/2}e^{-\pi(n/x)^2},
\end{equation}
which is essentially the twisted $\theta$-function.
Note here that we may calculate $\chi_\Delta(n)$ for any given $n$, even
without knowledge of $\Delta$; in fact, we have $\chi_\Delta(n)=\chi_d(n)$
unless $n$ has a common factor with $\ell$. We may assume without loss
of generality that we never come across such an $n$, since otherwise
we will have found a square factor of $d$, answering the original
question.

If one thinks of the character values $\chi_\Delta(n)$ as ``random''
$\pm1$ then, thanks to the decay of the Gaussian, the series in
\eqref{thetaseries} is the result of a random walk of length about $x$,
which will typically have size on the order of $\sqrt{x}$;
thus, one might expect $S_\Delta(x)$ to oscillate,
without growing very large or decaying, as $x\to\infty$.
This turns out to be an accurate description for $x$ up to $\sqrt{|\Delta|}$,
but for larger $x$, $S_\Delta(x)$ is constrained by the symmetry
\begin{equation}
\label{SDeltasymmetry}
S_{\Delta}(x)=S_{\Delta}(|\Delta|/x),
\end{equation}
following from the Poisson summation formula (see
\cite[pp.~13, 68, 70]{davenport}).

The point of symmetry of \eqref{SDeltasymmetry} gives an indication
of $|\Delta|$, and thus we can rule out small values of $|\Delta|$
essentially by drawing the graph of $S_\Delta(x)$.  More precisely,
for any given $B>0$, one can decide whether or not $|\Delta|\le B$ in
time\footnote{We omit the proof of this, but the main point is the fact that
$S_\Delta(e^x)$ is the Fourier transform of the complete
$L$-function $\Lambda(\frac12+it,\chi_\Delta)$, so it is essentially
band-limited.}
$O_\varepsilon\bigl(N^\varepsilon\sqrt{B}\bigr)$, which matches the
running time of Pollard--Strassen for the same task.  Moreover, if one
could find a method of computing the $\theta$-function $S_\Delta(x)$
substantially more quickly than by direct in-order summation, say in
time $O_\varepsilon\bigl(N^{\varepsilon}x^{1-\delta}\bigr)$ for some
$\delta\in(0,1)$, then this improves to $O_\varepsilon\bigl(N^\varepsilon
B^{\frac12(1-\delta)}\bigr)$; in particular, taking
$B=N^{\frac1{3-2\delta}}$ and falling back on Pollard--Strassen
to rule out $\ell\le\sqrt{N/B}$,
we would get an algorithm to certify $N$ squarefree in time
$O_\varepsilon\bigl(N^{\frac16(1-\frac{\delta}{3-2\delta})+\varepsilon}\bigr)$.

\section{The explicit formula}
\label{explicit}
Our main interest, however, is in algorithms that work in subexponential
time. This is difficult to attain in the above approach because we used
sums over integers $n$. It is well-understood in problems of this type
that one can do better by considering sums over primes, at the expense
of having to assume GRH.

To be precise, let
$L(s,\chi_\Delta)=\sum_{n=1}^\infty\chi_\Delta(n)n^{-s}$ be the Dirichlet
$L$-function corresponding to $\Delta\ne1$.  Assuming GRH,
the non-trivial zeros of $L(s,\chi_\Delta)$ may be written
as $\frac12\pm i\gamma_j(\Delta)$, $j=1,2,3,\ldots$, where
$0\le\gamma_1(\Delta)\le\gamma_2(\Delta)\le\ldots$, and each
ordinate is repeated with the appropriate multiplicity.\footnote{If
$L(s,\chi_\Delta)$ has a zero at $s=\frac12$ of multiplicity $m$,
then $m$ is necessarily even, and we take $\frac{m}2$ copies
of this zero, i.e.\ $\gamma_j(\Delta)=0$ for $j\le\frac{m}2$ and
$\gamma_{\frac{m}2+1}(\Delta)>0$.}  Further, let $g:[0,\infty)\to\C$
be a test function which is continuous of compact support, piecewise
smooth, and has cosine transform $h(t)=2\int_0^\infty g(x)\cos(tx)\,dx$.
Then the ``explicit formula'' for $L(s,\chi_\Delta)$ reads
\begin{equation}
\label{explicitformula}
\begin{aligned}
g(0)\log|\Delta|&=2\sum_{j=1}^{\infty}h\bigl(\gamma_j(\Delta)\bigr)
+2\sum_{n=1}^{\infty}\frac{\Lambda(n)\chi_\Delta(n)}{\sqrt{n}}g(\log{n})
\\
&+\,g(0)\log(8\pi e^\gamma)
-\int_0^\infty\frac{g(0)-g(x)}{2\sinh(x/2)}\,dx
+\chi_\Delta(-1)\int_0^\infty\frac{g(x)}{2\cosh(x/2)}\,dx,
\end{aligned}
\end{equation}
where $\Lambda$ is the von Mangoldt function.

If not for the sum over zeros $\gamma_j(\Delta)$, this would be exactly
what we seek, i.e.\ a formula for the conductor $|\Delta|$ in terms
of character values. Without knowledge of the zeros, we do not get such
an exact identity, but we can at least get an inequality in one
direction if the test function is chosen so that $h$ is non-negative,
i.e.\
\begin{equation}
\label{lowerbound}
\begin{aligned}
\log|\Delta|\ge
2&\sum_{n=1}^{\infty}\frac{\Lambda(n)\chi_\Delta(n)}{\sqrt{n}}g(\log{n})
+\log(8\pi e^\gamma)\\
&-\int_0^\infty\frac{1-g(x)}{2\sinh(x/2)}\,dx
+\chi_\Delta(-1)\int_0^\infty\frac{g(x)}{2\cosh(x/2)}\,dx,
\end{aligned}
\end{equation}
for any $g:[0,\infty)\to\R$ which is continuous of compact support,
satisfies $g(0)=1$, and has non-negative cosine transform.\footnote{A
function $g$ satisfying these conditions need not be piecewise smooth,
and in fact $\int_0^\infty\frac{1-g(x)}{2\sinh(x/2)}\,dx$ may
be divergent.
However, since $g$ has non-negative cosine transform,
$\frac{1-g(x)}{2\sinh(x/2)}$ is non-negative, so we may interpret the
right-hand side of \eqref{lowerbound} as $-\infty$ whenever the integral
diverges. With that convention, a standard approximation argument shows
that \eqref{lowerbound} holds for all $g$ as indicated.}

In the next few subsections we explore some strategies for exploiting
\eqref{lowerbound} to prove that our given $d$ is squarefree.
The proofs of 
Propositions~\ref{optimalprop}--\ref{runningtimeprop}
below are given in the appendix.

\subsection{Varying the test function}
\label{varytestfunction}
Our first, and simplest, strategy is to search for a test function such
that the right-hand side of \eqref{lowerbound} is close to $\log|\Delta|$.
Naturally, we pay a price for ignoring the zero sum $Z=\sum_{j=1}^\infty
h(\gamma_j(\Delta))$, in that our estimate for $|\Delta|$ is a factor of
$e^{2Z}$ too small. We can still use this to prove that $d$ is squarefree
by ruling out values of $\ell\le e^Z$ using Pollard--Strassen or otherwise,
but this takes exponential time $\gg e^{Z/2}$ in the size of $Z$.

On the other hand, note that the sum over prime
powers in \eqref{lowerbound} is exponentially long, i.e.\ if $g$
has support $[0,X]$ then we need to compute the right-hand side of
\eqref{lowerbound} for $n$ up to $e^X$.  Thus, we would like $X$ not to
be very large.  However, if we choose $X$ too small then, by the
uncertainty principle, the cosine transform $h$ will be relatively
``wide'', so that the zero sum will typically be large.

Our first result shows that there is an optimal choice of test function
for each fixed $X$, and thus an optimal tradeoff between these two
exponential penalties.
\begin{proposition}
\label{optimalprop}
Let $\mathcal{C}(X)$ be the class of functions $g:[0,\infty)\to\R$
that are continuous, supported on $[0,X]$, have non-negative
cosine transform, and satisfy $g(0)=1$. For $g\in\mathcal{C}(X)$, let
$l(g)$ denote the right-hand side of \eqref{lowerbound}.
Then for every $X>0$ there exists $g_X\in\mathcal{C}(X)$ such
that $l(g_X)\ge l(g)$ for all $g\in\mathcal{C}(X)$.
\end{proposition}
We remark further that if $g\in\mathcal{C}(X)$ then its cosine
transform $h$ is band-limited, and so, by Jensen's formula, $h$ has at
most $O_X(T)$ zeros in the interval $[-T,T]$ for large $T$ (see
\cite[p.~16]{martin}).
Since it is known that $L(s,\chi_\Delta)$ has $\gg T\log{T}$ distinct
zeros with imaginary part in $[-T,T]$, under GRH the zero sum in
\eqref{explicitformula} cannot vanish, so
that \eqref{lowerbound} is a strict inequality for any fixed $X$,
i.e.\ $\log|\Delta|>l(g_X)$. However, it is easy to see that
$l(g_X)$ tends continuously and monotonically to $\log|\Delta|$
as $X\to\infty$.

Although Prop.~\ref{optimalprop} is an existence result only, one can try to
solve for the optimal test function $g_X$ by approximating $\mathcal{C}(X)$
using a sufficiently rich, finite-dimensional space of functions.
For instance, let $M$ be a non-negative integer, and consider step functions $f$
of the form
\begin{equation}
\label{fx}
f(x)=\sum_{n=-M}^M
a_n\mathbf{1}_{(-1/2,1/2)}\!\left(\frac{2M+1}{X}x-n\right)
\quad\text{for }x\in\R,
\end{equation}
for arbitrary real coefficients $a_n$. If we take $g$ to be the
autocorrelation of $f$, i.e.\ $g(x)=\int_\R f(y)f(x+y)\,dy$,
then $g$ has cosine transform $\bigl|\hat{f}(t)\bigr|^2\ge0$,
the right-hand side of \eqref{lowerbound} is a
quadratic form in the $a_n$, and the condition $g(0)=1$ amounts to an
$L^2$-normalization. 
Thus, we can find the optimal lower bound for this
family of test functions by computing the matrix of the form and finding
its largest eigenvalue.\footnote{If $A$ is the matrix associated with the quadratic form and 
$c := (2M+1)/X$, then $\max_{|a|^2= c} a^t A a = c \lambda_1,$ where $\lambda_1$ 
is the largest eigenvalue of $A$, $a := (a_{-M},\ldots,a_M)$, and the 
condition $|a|^2 = c $ is equivalent to $g(0)=1$.} 
It is not hard to see that this family comes
arbitrarily close to the optimal $g_X$ as $M\to\infty$, although it may
be the case that $g_X$ is highly oscillatory, meaning that we would need
to take $M$ very large before finding a close approximation to it.

\subsection{Twisting}
\label{twistsection}
A second strategy, which performs well in practice,
is to ``twist'' our given quadratic character $\chi_d$
by other characters $\chi_q$, and look for a $q$ for which the lower bound
in \eqref{lowerbound} is favorable. This is related to the first
strategy since, by Fourier analysis, varying the test function amounts to
considering combinations of the twists by $n^{it}$ for various $t$.
Twists by quadratic characters have the added advantage of zero repulsion around
the central point, as we explain in detail in \S\ref{rmtcalc}.

In other words, if we run out of luck with our given
value of $d$ then we can multiply it by $q\in\mathcal{F}$ relatively
prime to $d$ and ask if the product is a fundamental discriminant. This
operation also introduces a penalty, since \eqref{lowerbound} becomes
a lower bound for $\log|q\Delta|$, so we have to subtract $\log|q|$:
\begin{equation}
\label{lowerbound2}
\begin{aligned}
\log|\Delta|\ge -&\log|q|+
2\sum_{n=1}^{\infty}\frac{\Lambda(n)\chi_{q\Delta}(n)}{\sqrt{n}}g(\log{n})
+\log(8\pi e^\gamma)\\
&-\int_0^\infty\frac{1-g(x)}{2\sinh(x/2)}\,dx
+\chi_{q\Delta}(-1)\int_0^\infty\frac{g(x)}{2\cosh(x/2)}\,dx.
\end{aligned}
\end{equation}

What we gain by this strategy is the hope of finding a twist
$\chi_{q\Delta}$ such that the low-lying zeros of $L(s,\chi_{q\Delta})$
are unusually sparse, so that the zero sum $\sum_{j=1}^\infty
h(\gamma_j(q\Delta))$ can be made small even with a relatively simple
choice of $g$.  For instance, we might hope that $L(s,\chi_{q\Delta})$
has a large zero gap around the central point. In that case, we have
the following.
\begin{proposition}
\label{testfunctionprop}
Suppose that $L(s,\chi_{q\Delta})$ satisfies GRH and
has no non-trivial zeros
with imaginary part in $(-\delta,\delta)$.
Set $X=2\delta^{-1}(A+\log\log|q\Delta|)$ for some $A\ge0$.
Then there is an explicit $g\in\mathcal{C}(X)$ whose cosine transform $h$
satisfies
\begin{equation}
\label{zerosumbound}
\sum_{j=1}^{\infty}h(\gamma_j(q\Delta))
\ll\frac{e^{-A}X}{(\log\log|q\Delta|)^{3/2}},
\end{equation}
with an absolute and effective implied constant.
\end{proposition}
In other words, there is a test function $g$ with support of
size inversely proportional to the size of the zero gap for which the
zero sum is relatively small. Thus, ruling out small values of $\ell$ to
complete the proof that $d$ is squarefree is fast compared to evaluating
the explicit formula.\footnote{In fact, as a by-product of evaluating
the explicit formula, we will test $d$ for divisibility by all
primes $p<e^X$. Thus, if $\sum_j h(\gamma_j(q\Delta))<X$ then no
additional work is necessary to prove that $d$ is squarefree.}

Although it is difficult to ascertain directly for
a given $q$ whether $L(s,\chi_{q\Delta})$ has a large zero gap, we can
simply try computing the lower bound \eqref{lowerbound2} using the test
function given by Prop.~\ref{testfunctionprop} for a particular desired
value of $\delta$. We may repeat this procedure for many $q$ until we
find one which is good enough, and then use the quadratic form approach
with a relatively small matrix to refine the choice of test function.

The crucial question is thus how large of a zero gap can one expect to
find by searching through various $q$. On average, one expects the first
zero gap around the central point to be about $2\pi/\log{|q\Delta|}$,\footnote{
More precisely, the Random Matrix model for the family of $L$-functions in question suggests that the mean of the first zero gap should be this quantity multiplied by a constant whose value is approximately 0.78, see 
\cite{katz-sarnak,rubinstein}} which is of no use in Prop.~\ref{testfunctionprop}.
On the other hand,
if we found $q$ of modest size for which the first zero gap
was on the order of $1/\sqrt{\log|q\Delta|}$, say,
then we would have a fast
algorithm for proving that $d$ is squarefree.\footnote{If there
is a constant $\varepsilon>0$ such 
that one can always find a
$\gamma_1(q\Delta) \ge (\log\log |\Delta|)^{1+\varepsilon}/\log|\Delta|$, then
Prop.~\ref{testfunctionprop} already allows one to certify that an integer is
squarefree in subexponential time (on the GRH). It would be interesting
to see if one could push this line of thought to an improvement of the
$O(N^{1/6+o(1)})$ time bound of Pollard--Strassen,
but we do not do so here.}

To make this more precise,
anticipating a subexponential running time
on the order of $\exp\bigl((\log|\Delta|)^\theta\bigr)$, for $\theta>0$ we define
\begin{align*}
M_\Delta(\theta)&=
\max\Bigl\{\gamma_1(q\Delta):q\in\mathcal{F}, (q,\Delta)=1,
|q|\le\exp\bigl((\log|\Delta|)^\theta\bigr)\Bigr\},\\
\eta_\Delta(\theta)&=-\frac{\log M_\Delta(\theta)}{\log\log|\Delta|},
\quad
\eta_\infty(\theta)
=\limsup_{\substack{\Delta\in\mathcal{F}\\|\Delta|\to\infty}}
\eta_\Delta(\theta),
\quad
\theta^*=\inf\{\theta>0:\eta_\infty(\theta)\le\theta\}.
\end{align*}
Thus, $\eta_\Delta$ is a logarithmic measure of the largest gap
size that we encounter among the twists by $q\in\mathcal{F}$ with
$|q|\le\exp\bigl((\log|\Delta|)^\theta\bigr)$, with
$\eta_\Delta=1$ corresponding to an average gap, and $\eta_\Delta<1$
corresponding to larger gaps. Since, \emph{a priori}, we have no
information about our given discriminant, we take the worst case,
$\eta_\infty$, over all large values of $\Delta$.
Finally, $\theta^*$ measures the point at which the size of the
twisting set matches the expected length of the prime
sum that we need to evaluate in Prop.~\ref{testfunctionprop}.
Combining this with a brute-force search strategy,
we obtain the following.
\begin{proposition}
\label{runningtimeprop}
Assume GRH for quadratic Dirichlet $L$-functions.
There is an algorithm that takes as input a positive integer $N$ and outputs
either a non-trivial square factor of $N$ or a proof that
$N$ is squarefree. If $N$ is squarefree then the algorithm runs in time
$O\bigl(\exp[(\log{N})^{\theta^*+o(1)}]\bigr)$.
\end{proposition}
Note that the assumption of GRH in the proposition applies to the
certificates generated by the algorithm as well as its running time
analysis. 

\subsection{Examples}
In Figure ~\ref{largegapexample}, we give a basic illustration of the 
 favorable situation of a large gap around the central point.
We chose $L(s,\chi_d)$, where $d=1548889$ is a 
fundamental discriminant, because it has a gap 
size $\approx 1.747424$ there, which is about $4.5$ times
the average $0.78 \times 2\pi/\log (d/(2\pi))$. Therefore, we expect
the lower bound \eqref{lowerbound} to be quite good 
even with a simple choice of $g$. 
For instance, if
$M=0$ and $a_0=1/\sqrt{X}$ in \eqref{fx}, we obtain $g(x)=\max(0,1-|x|/X)$
and $h(t)=X \sin^2(Xt/2)/(Xt/2)^2$. Choosing $X=7/2$, we have
$2\sum_{j\ge 1} h(\gamma_j(d)) \approx 6.73$ (by computing the zeros explicitly using {\tt lcalc}), and so 
the lower bound \eqref{lowerbound} would be $\log d-6.73\approx 7.5$.
This would have sufficed to prove that $d$ is squarefree, since in computing 
the prime sum of the explicit formula we would have checked that $d$ 
has no factor $\le e^{7/2}$, and so certainly no factor $\le e^{6.73/2}$. 
In particular, \eqref{lowerbound} allows one to certify that $d$ is squarefree 
from the primes $\le e^{7/2} \approx 33$ only, which is
is already better than trial division. Thus, our strategy can lead to a gain 
even for small $d$.
\begin{figure}[!ht]
\includegraphics[scale=0.33, trim = 0 15 0 30, clip]{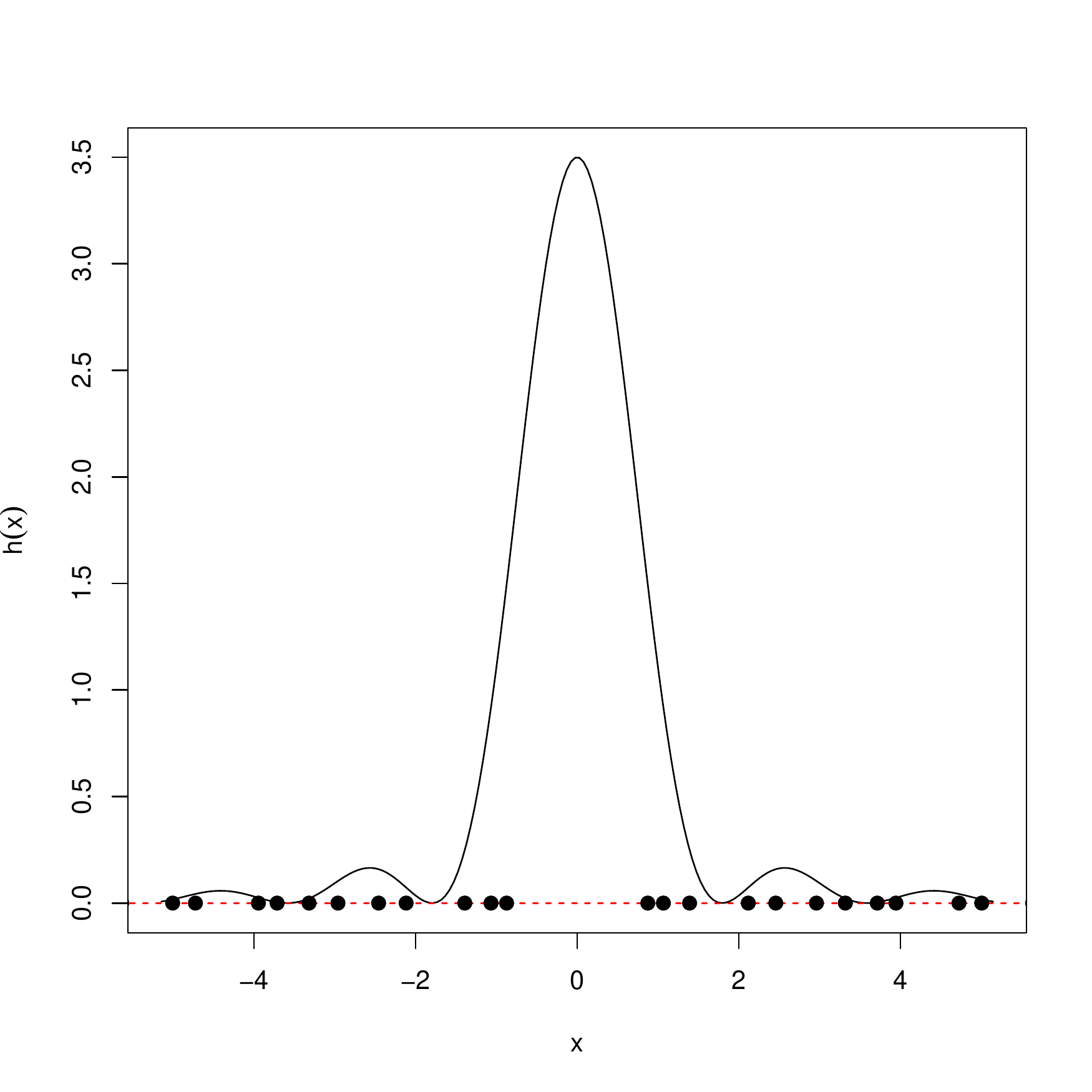}
\caption{Large zero gap around the central point of $L(s,\chi_{1548889})$, 
together with the test function on the zero side 
$h(t)=8\sin^2(7t/4)/(7t^2)$ resulting from $M=0$ and $X=7/2$.}
\label{largegapexample}
\end{figure}

The behavior of the lower bound as $X$ increases
is worth noting. We illustrate it for $L(s,\chi_d)$,
Figure~\ref{lowerboundexamples} (left plot), using the same simple 
choice of $g$ as
before. The overall shape of the plot is typical for the case of 
a large gap, in that there is an initial
(good) region where the lower bound increases steeply, 
followed by an inevitable, unless $L(1/2,\chi_d)=0$, region of small 
oscillations. If the gap about the center is 
not particularly large, however, then the
initial good region will be much smaller.
This is illustrated in Figure ~\ref{lowerboundexamples} 
(right plot) using the $L$-function of a randomly chosen fundamental discriminant, 
$L(s,\chi_{2000005})$, which has an average-sized gap of $\approx 0.515984$
about the center.
Notice that there is a wide good region later on in the plot, but 
it comes in too late to be useful in our algorithm. 
The main point is that the absence of zeros near $s=1/2$ allows the sum over prime powers to capture
the bulk of the r.h.s.\ of the explicit formula 
\eqref{explicitformula} with a smaller choice of $X$ (i.e. a more compactly supported $g$,
and slower decay for $h$ on the zeros sum).
\begin{figure}[!ht]
\centering
\includegraphics[scale=0.33, trim = 25 15 0 20,clip]{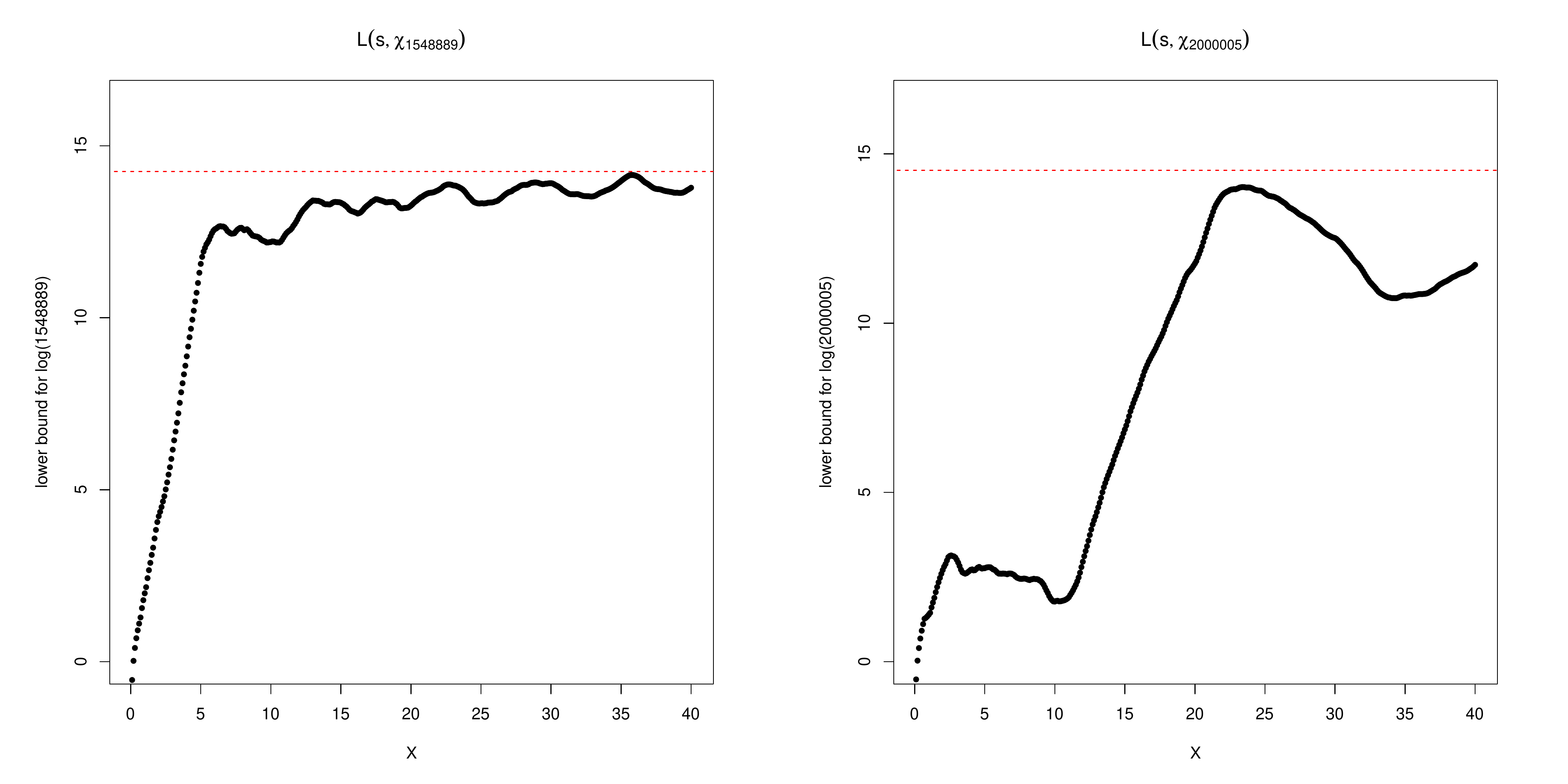} 
\caption{Behavior of the lower bound \eqref{lowerbound} 
as $X$ increases: The case of a large gap (left) compared 
with the case of an average gap.}
\label{lowerboundexamples}
\end{figure}

\section{Complexity}
By computing the $1$-level density of the family of twists by $\chi_q$,
$q\in\mathcal{F}$, 
one can see that $\eta_\infty(\theta)\le1$ for $\theta>1$, so that
$\theta^*\in[0,1]$. However, the algorithm of Prop.~\ref{runningtimeprop}
is subexponential only if
$\theta^*<1$, which unfortunately seems beyond the current technology
to prove, even under GRH.
We can, however, make a reasonable conjecture of the value of $\theta^*$
by answering the analogous question for a suitable random matrix
model, where the calculation is more tractable:
\begin{conjecture}
\label{etaconj}
We have
$$\eta_\infty(\theta)=\begin{cases}
1-\frac{\theta}2&\text{if }0<\theta<1,\\
\frac12&\text{if }\theta\ge1.
\end{cases}$$
In particular, $\theta^*=\frac23$.
\end{conjecture}
Thus, by Prop.~\ref{runningtimeprop},
we conjecture that our algorithm is capable
of certifying $N$ squarefree in time
$O\bigl(\exp[(\log{N})^{2/3+o(1)}]\bigr)$.

We give a detailed justification for the conjecture in \S\ref{rmtcalc}
below. First, however,
it turns out that one can arrive at the same conclusion for the running
time without any consideration of the zero sum in \eqref{explicitformula},
by analyzing the lower bound \eqref{lowerbound}
using a simple model of the $\chi_{q\Delta}(p)$ 
as independent random variables assuming the values
$1$ and $-1$ with equal probability (this is not always a good model \cite{rubinstein-sarnak} but suffices for our 
purposes).
We make this more precise in the following proposition,
which is a consequence of \cite[Theorem 1]{montgomery-odlyzko}. 
\begin{proposition}
\label{randomprimemodel}
Let $Y_1,Y_2,\ldots$ be independent random variables such that
$\mathbb{P}(Y_j=1)=\mathbb{P}(Y_j=-1)=\frac12$,
and put $Y := 2\sum_{p_j\le e^X}\frac{Y_j\log p_j}{\sqrt{p_j}}\left(1-\frac{\log
p_j}{X}\right)$, where $p_j$ denotes the $j$th prime number.
Then, for each $n$ satisfying $3\le n<e^X$, we have
$$
\mathbb{P}(Y \ge v_n) \ge 2^{-22} \exp\left(-\frac{30v_n^2}{c_n}\right), 
\qquad\mathbb{P}(Y \ge u_n) \le \exp\left(-\frac{u_n^2}{32c_n}\right),
$$
where $v_n:=\sum_{p_j\le n}\frac{\log p_j}{\sqrt{p_j}}\left(1-\frac{\log
p_j}{X}\right)$, $u_n:=4v_n$, 
and $c_n:=\sum_{n<p_j\le e^X} \frac{\log^2 p_j}{p_j}\left(1-\frac{\log
p_j}{X}\right)^2$. 
\end{proposition}
In particular, as $n, X\to\infty$ with $n=e^{o(X)}$, so that
$v_n\sim 2\sqrt{n}$ and $c_n\sim \frac{1}{12}X^2$, we get
$\mathbb{P}(Y\ge 2\sqrt{n})
\ge\exp(-(1440+o(1))n/X^2)$. Therefore, after 
$\lfloor\exp{X}\rfloor$ independent samples of $Y$,
we expect to occasion $Y\gtrsim \frac{1}{6\sqrt{10}}X^{3/2}$ at least once.
In the opposite direction, we have $\mathbb{P}(Y\ge 8\sqrt{n})
\le \exp(-(24+o(1))n/X^2)$, and so 
after $\lfloor\exp{X}\rfloor$ independent samples, we expect 
at most one instance of $Y\gtrsim \frac{4}{\sqrt{6}}X^{3/2}$.
Together, these estimates are consistent with $\theta^*=2/3$. 
Of course, Prop.~\ref{randomprimemodel}
simplifies the situation by ignoring the higher prime powers,
but that is not important since they contribute only $O(X)$, and so 
do not impact the $X^{3/2}$ term. It is worth noting, however, that 
the contribution of the higher prime
powers in numerical computations is still noticeable because $\chi_q(p^2)=1$
whenever $(q,p)=1$, and so the bulk of their contribution is guaranteed
to help our lower bound, regardless of the number of samples.

\subsection{A conjecture for $\theta^*$ via random matrix theory}
\label{rmtcalc}
The random matrix philosophy suggests (e.g.\ by comparing the 1-level densities)
that the relevant symmetry for a family of primitive quadratic 
twists is symplectic. The symplectic group $USp(2N)$ is a compact group consisting 
of $2N\times 2N$ unitary matrices $A$ satisfying $A^t J A = J$, where
$$
J := \left( \begin{array}{cc}
0 & I_N \\
-I_N & 0 \end{array} \right).
$$
The eigenvalues of $A$ lie on the unit circle,
come in conjugate pairs, and can be written uniquely as 
$$
e^{\pm i\theta_1(A)},\ldots,e^{\pm i\theta_N(A)},\qquad
0\le \theta_1(A) \le \cdots \le \theta_N(A) \le \pi. 
$$
Making the identification 
$2N = \log |\Delta|$,\footnote{This identification is obtained by equating 
the mean spacing of eigenphases of $A\in USp(2N)$, which is $\pi/N$, 
and the mean spacing of zeros of
$L(s,\chi_{q\Delta})$ at a fixed height, which is $\sim 2\pi/\log |q\Delta|\sim
2\pi/\log|\Delta|$ as $|\Delta|\to\infty$, $\Delta\in\mathcal{F}$.}
we expect that statistics of the lowest eigenphase $\theta_1(A)$ 
as $A$ varies in $USp(2N)$
coincide to leading order, and modulo arithmetic effects, 
with statistics of the lowest zero $\gamma_1(q\Delta)$ 
as $q$ varies in $\mathcal{F}$ but 
still sufficiently small compared to $|\Delta|$. 
Thus, by computing statistics of $\theta_1(A)$, we arrive at
conjectures for $\gamma_1(q\Delta)$.
In particular, since the complexity of our algorithm depends on 
the frequency of large values 
of $\gamma_1(q\Delta)$, we are led to consider the tail distribution of 
$\theta_1(A)$. 

To this end, and 
to facilitate comparison with other symmetry groups later on, 
let $U(N)$ denote the (compact) group of $N\times N$ 
unitary matrices, and $SO(2N) \subset U(2N)$ the group of 
orthogonal matrices of determinant $1$. The eigenphases of $A\in U(N)$
can be written uniquely as $0\le \theta_1(A) \le \cdots \le \theta_N(A) < 2\pi$,
while those of $A\in SO(2N)$, which come in pairs $\pm\theta_j(A)$, 
can be written uniquely as $0\le  \theta_1(A)\le\ldots \le  \theta_N(A) \le \pi$.
Let $\mathbb{P}_{G(N)}$ denote the unique Haar measure on $G(N)
\in\{U(N),SO(2N),USp(2N)\}$,  
normalized to be a probability measure. The random matrix philosophy
suggests, for example, that the relevant symmetry group for averages over 
a family of twists by $n^{it}$ is unitary, while for averages over a family of 
elliptic curves it is orthogonal (even or odd, depending on 
the sign of the functional equation in the family).

Let $\mathbb{P}^{\times M}_{G(N)} = \mathbb{P}_{G(N)}\times \cdots \times 
\mathbb{P}_{G(N)}$, repeated $M$ times, be the product measure on $G(N)^M$.
For each Borel-measurable set $J \subset [0,\sigma \pi]$, where $\sigma=2$ 
if $G(N)=U(N)$ and $\sigma=1$ otherwise,  
define 
$\mathcal{S}(J):=\{(A_1,\ldots, A_M) \in
G(N)^M: \max_{1\le m \le M} \theta_1(A_m) \in J\}$.
For short-hand,
we write $\mathbb{P}_{G(N)}(\max_{1\le m \le M} \theta_1(m) \in J)$
in place of $\mathbb{P}^{\times M}_{G(N)}(\mathcal{S}(J))$, 
$\mathbb{P}_{G(N)}(\max_{1\le m \le M} \theta_1(m)>s)$ 
in place of $\mathbb{P}^{\times M}_{G(N)}(\mathcal{S}((s,\infty)))$,
and so on. 
The distribution function $\mathbb{P}_{G(N)}\bigl(\theta_1>s\bigr)$ 
is known as the gap probability. 
The proofs of Propositions~\ref{uspgapprop}--\ref{ungapprop}
below are given in the appendix.
\begin{proposition}
\label{uspgapprop}
Fix $\beta \in (0,2)$, and define
$$
M_{\beta}(N) := \left\lfloor \exp((2N)^{\beta})\right\rfloor,
\quad s_{\varepsilon,\beta}^{\pm}(N):=(4 \pm \varepsilon)(2N)^{\beta/2 - 1}.
$$
Then, for each fixed $\varepsilon >0$, as $N\to \infty$ we have
$$
\mathbb{P}_{USp(2N)} \left(s^-_{\varepsilon,\beta}(N) < 
\max_{1\le m \le M_{\beta}(N)} \theta_1(m) \le
s^+_{\varepsilon,\beta}(N) \right) \to 1.
$$
In other words, $(2N)^{1-\beta/2}\max_{1\le m \le M_{\beta}(N)} \theta_1(m)$
converges in distribution to $4$.
\end{proposition}
Therefore, if we choose $A_1,\ldots,A_{M(N)} \in USp(2N)$, 
independently and uniformly with
respect $\mathbb{P}_{USp(2N)}$, then in the limit as $N\to\infty$, we have
$\max_{1\le m \le M(N)}
\theta_1(A_m) > s^-_{\varepsilon,\beta}(N)$ with probability approaching $1$. 
For instance, if $\beta=1$, then we expect to find at least one lowest eigenphase of size
$\gtrsim (4-\varepsilon)/\sqrt{2N}$.
Since the eigenvalues of symplectic matrices come in conjugate pairs, 
this corresponds to an eigenphase spacing 
$\ge 2(4-\varepsilon)/\sqrt{2N} \approx \sqrt{32/N}$, 
which is $\frac4{\pi}\sqrt{2N}$ times the average spacing. 

In contrast, the eigenvalues of unitary matrices 
do not necessarily come in conjugate pairs,  
so the point $1$ on the unit circle is no longer distinguished,
and actually $\mathbb{P}_{U(N)}$ is rotationally invariant. 
Thus, it is more natural to consider the nearest-neighbor
distribution function, 
$\mathbb{P}_{U(N)}(\theta_2-\theta_1 > u) := \mathbb{P}_{U(N)}(\{A\in U(N):
\theta_2(A)-\theta_1(A)>u\})$, which is related to the gap probability by 
differentiation; see \eqref{ungapprobderivative} in 
the appendix. In fact, $\log \mathbb{P}_{U(N)}(\theta_2-\theta_1 > u) \sim 
\log \mathbb{P}_{U(N)}(\theta_1 > u)$ as $N\to\infty$ over a wide range of $u$.
To facilitate comparison with the symplectic case, we consider
the half-spacings $\frac12[\theta_2-\theta_1]$ in the following.
\begin{proposition}
\label{ungapprop}
Fix $\beta \in (0,2)$, and define
$$
M_{\beta}(N) := \left\lfloor \exp(N^{\beta})\right\rfloor,
\quad u^{\pm}_{\varepsilon,\beta}(N):=\sqrt{8\pm\varepsilon}N^{\beta/2-1}.
$$ 
Then, for each fixed $\varepsilon\in(0,8]$, as $N\to \infty$ we have
$$
\mathbb{P}_{U(N)} \left(u^-_{\varepsilon,\beta}(N) < 
\max_{1\le m \le M_{\beta}(N)} 
\tfrac12[\theta_2(m)-\theta_1(m)] \le  u^+_{\varepsilon,\beta}(N) \right)
\to 1.
$$
In other words, $N^{1-\beta/2}\max_{1\le m \le M_{\beta}(N)} 
\frac12[\theta_2(m)-\theta_1(m)]$ converges in
distribution to $\sqrt{8}$. The same is true if, in addition, 
one maximizes over the spacings of each matrix, 
replacing $\frac12[\theta_2(m)-\theta_1(m)]$
by $\max_{1\le j \le N}\frac12[\theta_{j+1}(m)-\theta_j(m)]$,
where $\theta_{N+1} := 2\pi+\theta_1$.
\end{proposition}
In light of Prop.~\ref{ungapprop}, we see that sampling from $U(2N)$ 
does not do as well as sampling from $USp(2N)$.
For if we choose $\lfloor\exp((2N)^{\beta})\rfloor$ matrices 
from $U(2N)$, independently and uniformly with respect to $\mathbb{P}_{U(2N)}$, 
then we expect that half the max spacing is $\approx
\sqrt{8}(2N)^{\beta/2-1}$, which is worse
than the symplectic case by a factor of $\sqrt{2}$. 
Note that we could have compared $USp(2N)$ and $U(N)$ instead,
but to do so meaningfully the eigenphases in the two ensembles
should be re-normalized to have the same mean spacing,
so that $U(N)$ still does worse by a factor of $\sqrt{2}$.

This suggests that our algorithm should do better if
it searches through quadratic twists rather than twists by $n^{it}$, 
i.e.\ $\gamma_1(q\Delta)$ as opposed to
$\frac12[\gamma_{j+1}(\Delta)-\gamma_j(\Delta)]$, for $q$ and $j$ 
in a suitable range.\footnote{To clarify the analogy with $U(N)$ a little more, 
we expect
$\max_{t\le \gamma_j(\Delta) < t+2\pi} 
\frac12[\gamma_{j+1}(\Delta) - \gamma_j(\Delta)]$, for $t=|\Delta|^{o(1)}$,
to be modelled by $\max_{1\le j\le N} \frac12[\theta_{j+1}(m)-\theta_j(m)]$,
where $N\approx \log |\Delta|$.}
This agrees with our observations in practice.
An additional reason for it might be that the assumption of
independent samples is less applicable to the gaps
$\gamma_{j+1}(\Delta)-\gamma_j(\Delta)$,
since they come from a single $L$-function, and are thus
constrained by its analytic properties, in contrast to the
$\gamma_1(q\Delta)$, which come from different $L$-functions.
For example, we intuitively expect $\gamma_{j+1}(\Delta)-\gamma_j(\Delta)$ 
to have negative correlations over
short ranges (and also some long-range correlations due to the primes; 
see \cite{odlyzko} for a numerical discussion of this 
in the case of zeta). While such negative correlations do not affect 
the $2/3$ in the analogue of Conj.~\ref{etaconj}
for the $t$-aspect, they likely make the implied asymptotic constants worse.

In order to make a conjecture for $\theta^*$ based on our 
$USp(2N)$ calculation, 
we identify $2N$ with $\log|q\Delta|$, as usual, and 
the lowest eigenphase with $\gamma_1(q\Delta)$. 
If $\theta<1$ then twisting by $\chi_q$ does not affect the density of
zeros appreciably, so we may
interpret Prop.~\ref{uspgapprop} for fixed $2N\approx\log|\Delta|$
as sampling 
$\gamma_1(q\Delta)$ for $q$ from $\{q\in\mathcal{F}:(q,\Delta)=1, |q|\le
\exp((\log|\Delta|)^{\theta})\}$.
The conclusion of the proposition thus suggests that
$M_{\Delta}(\theta) \asymp (\log|\Delta|)^{\theta/2-1}$;
in particular, we expect $\eta_{\infty}(\theta) = 1-\theta/2$, 
and so $\theta^* = 2/3$. 
On the other hand, if $\theta>1$ then $q$ becomes the
dominating factor in the zero density; thus, we expect the maximum
of $\gamma_1(q\Delta)$ to be attained for a relatively small choice of $q$,
meaning we do not derive any benefit from increasing $\theta$ further,
and $\eta_\infty(\theta)$ is constant.
Note that similar conclusions are reached if 
we sample twists by $n^{it}$ instead.

Finally, we remark that
Conj.~\ref{etaconj} is of independent interest and may warrant further study.
One can try to confirm it directly (i.e.\ by computing the 
first zero of many twists),  but this requires taking $|\Delta|$ 
fairly large before one can hope to discern a clear
pattern. Basic experiments suggest taking $|\Delta| \gtrsim 10^{15}$, say,
which is prohibitively time-consuming using the standard approximate 
functional equation, as one would need to compute $\gamma_1(q\Delta)$ 
for millions of $q$. 
It would perhaps be better to formulate a precise conjecture for
$M_{\Delta}(\theta)$ itself, 
including lower order terms, and check the numerics for that. One could
also try to confirm the conjecture for other families.

\subsection{Numerical results}
\label{numerics}
We applied our method to several {\tt RSA}-numbers, but our main test case was
{\tt RSA-210}, which is the following 210-digit number:
\begin{align*}
\textrm{{\tt RSA-210}}=\;
&2452466449002782119765176635730880184670267876783327\\
&5974341445171506160083003858721695220839933207154910\\
&3626827191679864079776723243005600592035631246561218\\
&465817904100131859299619933817012149335034875870551067.
\end{align*}
We searched for candidate twists (i.e.\ twists that are expected to make the prime 
sum large) essentially by brute force, with some modest refinements
described in \S\ref{smallprimes}.
We first used a simple weighting function, 
such as a triangle wave, to
evaluate a short prime sum (typically with $p\le 10^4$) for all twists
within a given range, then
incrementally increased the length of the sum as we filtered the results.
The candidates found this way 
were then fed into the lower bound \eqref{lowerbound}, this time using 
a much longer prime sum and the test function produced by 
the quadratic form method outlined in
\S\ref{varytestfunction}.\footnote{Note that the
integral in \eqref{explicitformula} can be computed to high precision 
using standard numerical integration methods.
Moreover, in our final, long prime sum, we used approximations of the
test function by Chebyshev polynomials, which allows most of the
summation to be carried out in integer arithmetic. In this way we can
effectively control the round-off error in the computation. On the other
hand, since the longest sum that we computed was over the primes
$\le 2.5\times 10^{16}$, standard double-precision arithmetic would
also suffice to control the round-off errors effectively for many
choices of $g$, e.g.\ by using pairwise summation.}
Our best-performing twist was $-9334602088654580277283 = 
-568391 \times 2345033 \times 7003250461$, which yielded the lower bound
$\log|\Delta|\ge 137.5158$ using $X=\log(1.3\times 10^{15})$ and
$M=312$ in \eqref{fx}.

\begin{proof}[Proof of Theorem~\ref{numericalresult}]
We illustrate the proof for the case of $N=\textrm{{\tt RSA-210}}$;
the parameter choices that we used to complete the proof for the other
challenge numbers are summarized in Table~\ref{prooftable}.

Suppose, to the contrary, that all prime 
factors of $N$ have multiplicity $>1$. Then 
$N=s^2|\Delta|^3$, where $|\Delta|$ is the conductor of $\chi_{-N}$.
We verified that $N$ is not a perfect cube, and
our computation showed that $\log|\Delta| > 137.515$, so that
$1<s< 1.3\times10^{15}$.  
This leads to a contradiction since, as 
a by-product of computing the prime
sum in the explicit formula, we checked that 
$N$ has no non-trivial factor $<1.3\times 10^{15}$.
\end{proof}
\begin{table}[h!]
\begin{tabular}{r|r|r|r|r}
$N$ & $q$ & $e^X$ & $M$ & $\log|\Delta|\ge$\\ \hline
{\tt RSA-210} & $-9334602088654580277283$ & $1.3\times10^{15}$ & 312 & 137.5158 \\
{\tt RSA-220} & $970064118336081477109$ & $1.8\times10^{15}$ & 312 & 145.2599 \\
{\tt RSA-230} & $2298170792729446843801$ & $2.5\times 10^{16}$ & 312 & 150.8289\\
{\tt RSA-232} & $-2779263460367695431079$ & $10^{16}$ & 312 & 152.7847
\end{tabular}
\caption{Parameters used in the proof of Theorem~\ref{numericalresult}}
\label{prooftable}
\end{table}
\begin{remark}
We do not need the full strength of the bound $\log |\Delta| > 137.515$ 
to prove the theorem, as we separately ruled out factors
$\le 10^{20}$ using the implementation of Pollard's $p-1$
method\footnote{Like Pollard--Strassen, this method can be used to rule
out small factors, with comparable complexity.
Pollard--Strassen has a theoretical advantage, in that the $p-1$ method
produces an inconclusive result if a randomly-chosen residue happens
to have exactly the same order modulo every prime factor of $N$;
however, the chance of that occurring is vanishingly small, so this
is irrelevant in practice.} 
in GMP-ECM \cite{gmp-ecm}, which takes less than
a day on a computer with 80GB of memory.
Therefore, the bound $\log |\Delta| > 130.02$ suffices
to prove the theorem, which
reduces the size of the prime sum needed to $p\le 2.66\times 10^{14}$.
A similar improvement was noted for the other challenge numbers.
\end{remark}

\begin{figure}[!ht]
\includegraphics[scale=0.6,trim = 0 15 0 40, clip]{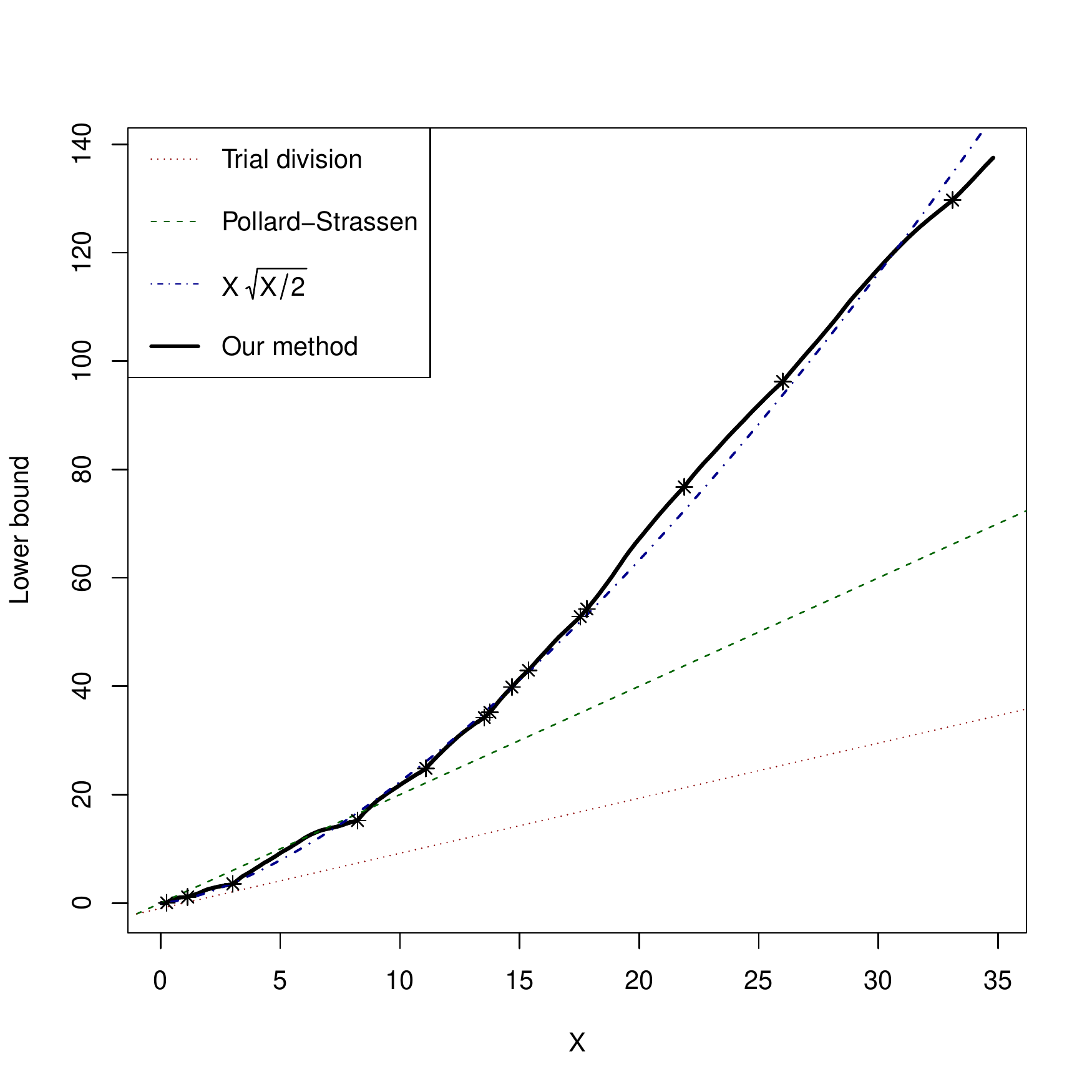}
\caption{The lower bound (under GRH) produced by our method when applied 
to {\tt RSA-210} using the primes $\le e^X$. The $\ast$s mark 
the places where the best performing twist available so far changes.
The slope increases noticeably at each $\ast$, except 
towards the end, where it is likely that we are not finding the best twists.
Also note that the
$\sqrt{2}$ in our fitted curve is likely not an absolute constant,
but varies as a small power of $\log\log|\Delta|$.}
\label{best1e15}
\end{figure}

In Figure~\ref{best1e15}, we present data about the practical efficiency of our 
algorithm, providing further evidence for the $2/3$ exponent.
To clarify the situation, 
recall that $\theta^*$ is chosen to balance the number of terms 
in the prime sum versus the number of twists we need to try
so that, with high probability, the zero contribution 
is small for at least one twist.
We accomplish this by taking $g$ with support inversely
proportional to the largest gap that we anticipate after trying
$\exp((\log|\Delta|)^{\theta})$ twists, and so our prime sum has length  $\le
\exp{(c_1(\log|\Delta|)^{\eta_{\Delta}(\theta)})}$ 
for some suitable constant $c_1>0$.
The $2/3$ arises as Prop.~\ref{uspgapprop} suggests that
$\eta_{\Delta}(\theta)\approx 1-\theta/2$, and on solving $1-\theta/2=\theta$.
More precisely, it arises since if we 
sample $\gamma_1(q\Delta)$ over $q\in\mathcal{F}$, $|q| < \exp{X}$, 
with $X$ much smaller than $\log |\Delta|$, then Prop.~\ref{uspgapprop}
suggests we should encounter at least one $\gamma_1(q\Delta) \ge 4 \sqrt{X}/\log|\Delta|$. 
Therefore, looking back at Prop.~\ref{testfunctionprop},
we expect to obtain a lower bound very close to $\log|\Delta|$ 
in time $\lesssim \exp(X+\frac{c_2 \log|\Delta|\log\log|\Delta|}{4\sqrt{X}})$, where 
$c_2>0$ is a constant implied by the proposition.
Optimizing, we choose $X = (\frac{c_2}{4}\log|\Delta|\log\log|\Delta|)^{2/3}$.

This reasoning on its own does not fully 
explain what we observe in Figure~\ref{best1e15}, 
which is that by sampling $\exp{X}$ twists
and using a prime sum of length $\exp{X}$, we seem to obtain 
a lower bound like $X^{3/2}$, even for intermediate 
values of $X$ much smaller than $(\log|\Delta|)^{2/3}$.
This behavior is expected 
if one treats the prime sum as a sum of independent random variables,
as in Prop.~\ref{randomprimemodel},
but it would be reassuring to see it from the zeros directly.
The difficulty towards this is that if $h$ does not have sufficient decay
outside the large gap, then we cannot bound 
the contribution of the zeros effectively (cf.\ Prop.~\ref{testfunctionprop}). 
Nevertheless, one can obtain a heuristic explanation, as follows. 

We choose $h$ with $0\le h(t)\le 1$, say, and mostly concentrated
within the interval $[-1/X,1/X]$, roughly speaking. 
We let $N_{\chi}(t):=\#\{0\le \gamma(\chi)<t\}$ denote the zero-counting function, 
and assume $L(1/2,\chi)\ne 0$ for simplicity. Then  
$\sum_{\gamma(\chi)} h(\gamma(\chi)) =2\int_0^{\infty} h(t) 
\,dN_{\chi}(t)$.
The contribution of the smooth part 
of $N_{\chi}(t)$ to the integral is $\sim g(0)\log|\Delta|$, which is precisely 
the left-hand side of \eqref{explicitformula}. Therefore, the 
prime sum contribution, which is basically our lower bound,
should be  $\approx -2\int_0^{\infty} h(t)\, dS_{\chi}(t)$, where
$S_{\chi}(t)$ is 
the fluctuating part of $N_{\chi}(t)$. This last integral is typically very small
due to the random nature of $S_{\chi}$, 
except we purposefully introduced a bias in it via our choice of twist, 
resulting in a large gap around the center
of size like $\sqrt{X}/\log|\Delta|$.
The contribution of this bias to the prime sum is essentially, for a reasonable 
$h$, $-2\int_0^{\sqrt{X}/\log|\Delta|} h(t)\, dS_{\chi}(t) \gg \sqrt{X}$. 
Since we expect the contribution from the interval
$[\sqrt{X}/\log|\Delta|,\infty)$  
to wash out in comparison,\footnote{This is the part 
of the heuristic that we cannot prove, even under the GRH, unless 
 $h$ has sufficient decay outside the zero gap.} we should get a lower bound 
like $g(0)\log|\Delta|\gg \sqrt{X}$. Finally, since $g(0) \ll 1/X$, 
we should get $\log|\Delta| \gg X^{3/2}$. 

This heuristic indicates how the
running time of our algorithm is controlled by
extreme (negative) values of $S_{\chi}(t)$. If $S_{\chi}(t) \ll_t
(\log|\Delta|)^{1/2+o(1)}$, for example, then we cannot expect a running time
better than $\exp((\log|\Delta|)^{1/2+o(1)})$, even if we
allow for an oracle supplying the algorithm with 
the best twist in any requested range. (This is in agreement with
Conj.~\ref{etaconj}.)
On the other hand, if $S_{\chi}(t)$ can get much larger (without
violating the GRH, so our method can still apply!),
then there is no such barrier.

\section{Refinements}
In this section, we describe a few refinements of our basic method and
indicate some directions for future research.

\subsection{Linear programming}
A natural question is whether one can make better use of the zero
sum in \eqref{explicitformula} than simply ignoring it by positivity,
as in \eqref{lowerbound}, especially since it typically dominates the
right-hand side when $X$ is small.  One idea is to apply the explicit formula
\eqref{explicitformula} with various choices of test function, setting
up a system of inequalities, and try to obtain a non-trivial lower
bound for the sum over zeros.  Since $\log|\Delta|$ also appears in
\eqref{explicitformula} and remains unknown to us, the logic of this may
seem circular at first glance, but we gain some additional information
coming from the fact that the zeros occur discretely, as we elaborate
below.

An immediate practical problem is that
the system involves infinitely many variables, since the zero sum is
infinite, and $h$ cannot be compactly supported (it has to be analytic).
Nevertheless, one can reduce to a finite number of variables, 
without too much loss, using an explicit estimate of the form 
$|\sum_{|\gamma| \ge T}h(\gamma)|=|2\int_T^{\infty} h(t)\,dN_{\chi}(t)|
\le \mathcal{E}(h,T)$, $T>0$, 
 simply bounding the conductor by the modulus, 
 and using known estimates for $S_{\chi}(t)$. Hence
\begin{equation}
\label{zerobounds1}
\sum_{|\gamma|<T} h(\gamma) - \mathcal{E}(h,T)\le
\sum_{\gamma} h(\gamma) \le \sum_{|\gamma| <T} h(\gamma) +
\mathcal{E}(h,T).
\end{equation}
A more serious problem is that the system is not
linear in the zero ordinates, and therefore is likely very unstable. 
We linearize the system, at the cost of having more variables or 
extra solutions, by subdividing the interval $[0,T)$ into bins
of size $\delta$, so that the variables become the count of 
zeros in each bin rather than the zeros themselves.
Specifically, for each integer $V>0$, and each integer $v\in[0,V)$,
let $\delta:=T/V$, 
$I(v):=[v\delta,(v+1)\delta)$, $m(v):=\frac{1}{2}\#\{\gamma:|\gamma| \in I(v)\}$,
$h^+(v):=\sup_{t\in I(v)} h(t)$, and $h^-(v):=\inf_{t\in I(v)} h(t)$. 
Then we have
\begin{equation}
\label{zerobounds2}
2\sum_{0\le v<V} m(v) h^-(v)\le
\sum_{|\gamma|<T} h(\gamma) \le 2\sum_{0\le v<V} m(v) h^+(v)\,.
\end{equation}
Applying \eqref{zerobounds1} and \eqref{zerobounds2} with a set of test
functions $\{(g_k,h_k):1\le k\le K\}$ of our choice, we 
obtain a linear system
\begin{equation}
\label{lpsystem}
\begin{aligned}
2\sum_{0\le v<V} m(v) h_k^-(v) -\mathcal{E}(h_k,T) 
\le g_k(0) \log|\Delta| + &g_k(0)\log q - P(g_k,q) \\
&\le 2\sum_{0\le v<V} m(v) h_k^+(v)+\mathcal{E}(h_k,T)
\end{aligned}
\end{equation}
for $k=1,\ldots,K$, where $\chi_q$ is the twist used, 
and $P(g_k,q)$ denotes the contribution from the prime sum and
integral terms in \eqref{explicitformula}.
Note that $V$ controls the size of
each bin, and $T$ controls the point where we truncate the zero sum. 
Finally, we let ${\tt logd}$ denote the unknown value of $\log|\Delta|$, and 
feed the system \eqref{lpsystem} into a linear programming solver,
such as {\tt GLPK} \cite{glpk}, 
with {\tt logd} as the objective function to be minimized.

We experimented with this approach for {\tt RSA-210} 
using various choices of $q$, $T$, $V$, and  
$\{(g_k,h_k):1\le k\le K\}$. 
For example, one of the better performing twists 
found, as in \S\ref{numerics}, was $q=-65123121667$.
Using this twist, we set up the system \eqref{lpsystem}
with $T=4$ and $V=500$ 
(so that $\delta=0.008$, which is smaller than the mean zero
spacing $\approx 0.013$), and 
$$
h_k(t) = \left[\frac{\sin(Xt/2k)}{(Xt/2k)}\right]^{2k},\qquad k=1,\ldots,7,
\qquad X=7\log 10, 
$$
so that $g_k(x)$, $k=1,\ldots,7$, are supported on $|x|\le X$.\footnote{The inequalities
in \eqref{lpsystem} were imposed in both directions except for $h_1$, where
only the lower bound was used.} We imposed
an integer variable constraint on $m(v)$, $v=0,\ldots,44$, with the rest 
being real variables. The integer variables are located at the beginning,
covering the interval $[0,0.36)$, which is reasonable since
$h_k(t)$ is not too small there and so detected zeros have more weight. 
Solving this system, we obtained
the lower bound $\log|\Delta| \ge 47.153$, of which $2.494$ came from the zeros.
This represents an improvement of about $5.5\%$ over using $\max_{1\le k\le 7}
[P(g_k,q)-g_k(0)\log q]$ alone, which is comparable to the improvement
that we obtained from using the Pollard $p-1$ algorithm to rule out
small values of $\ell$, as remarked after the proof of
Thm.~\ref{numericalresult}. Although this is a modest improvement on a
logarithmic scale, it makes
a substantial difference in the length of the final prime sum.

In general, further gains are possible by using
more integer variables, a smaller grid spacing (smaller $\delta$),
or additional test functions, in 
that order of importance. 
In reality, adding more test functions of compact support of size $X$  
loses impact quickly, which is not surprising because 
such functions cannot resolve zeros to better than $O(1/X)$.
Most of the gains, in fact, come from imposing integer constraints.
If no integer constraints are
imposed, the improvement in the above example 
goes down significantly, to around $1\%$.
Also, if all the variables are real, then the linear programming
approach is closely related to the approach of varying the test
function, described in \S\ref{varytestfunction}, and so cannot
be expected to do significantly better.\footnote{In 
the real variable case one can obtain an easily verifiable
certificate that the solution is indeed correct by solving the dual problem.
This is not available if one imposes integer constraints, and so one
has to trust the linear programming software in that case.}

However, one has to weigh the extra time 
it takes to set up and solve the mixed integer programming problem 
against the time it takes to simply compute a longer prime sum.
In the above example, it took about
$15$ minutes to solve the problem, but it can take 
much longer if more integer constraints are imposed.
It is tempting to think that if one could allow the number of integer
variables to grow very large
without significant time penalty then there would be
no limit to the improvement that could be obtained.
We offer the following theoretical evidence in favor of that belief.
\begin{definition}
Let $S=\{z\in\C:|\Im(z)|<1/2\}$.
A \emph{divisor} on $S$ is a function
$m:S\to\Z$ which is supported on a discrete subset of $S$.  A divisor
$m$ is \emph{admissible} if $m(-\gamma)=m(\gamma)\ge0$
for all $\gamma\in S$ and there is a number $A\ge0$ such that
$\sum_{\substack{\gamma\in S\\|\gamma|\le T}}m(\gamma)\ll T^A$ for
all $T\ge1$.
\end{definition}

\begin{proposition}
\label{converseprop}
Let $m:S\to\Z_{\ge0}$ be an admissible divisor,
$d\in\R^\times$, and $\{c_n\}_{n=2}^{\infty}$ a sequence of complex
numbers satisfying $c_n\ll n^{-\delta}$ for some $\delta>0$.
Suppose that for every smooth, even function
$g:\R\to\C$ of compact support and cosine transform $h$ we
have the equality
\begin{align*}
g(0)\log|d|&=\sum_{\gamma\in S}m(\gamma)h(\gamma)
+2\sum_{n=2}^{\infty}c_ng(\log n)\\
&+g(0)\log(8\pi e^\gamma)
-\int_0^\infty\frac{g(0)-g(x)}{2\sinh(x/2)}\,dx
+(\sgn{d})\int_0^\infty\frac{g(x)}{2\cosh(x/2)}\,dx.
\end{align*}
Then $d$ is a fundamental discriminant,
$c_n=\frac{\Lambda(n)\chi_d(n)}{\sqrt{n}}$
for every $n\ge2$, and $m(\gamma)=\ord_{s=1/2+i\gamma}L(s,\chi_d)$
for all $\gamma\in S$.
\end{proposition}
Thus, the explicit formula is rigid in the sense that the only identities
of the shape \eqref{explicitformula} that can hold for all test functions
are the ones arising from quadratic character $L$-functions.
We remark that the key to this proposition, whose full proof is given in
the appendix, is that $m$ is integer valued and supported on a discrete
set.  Unfortunately, the proposition is ineffective,
in that it does not predict how many or how complicated we must choose
the test functions before finding a system that yields a good lower
bound for $\log|d|$.
However, note that under GRH, the $\Delta\in\mathcal{F}$ with $|\Delta|\le x$
are distinguished from one another by the values of $\chi_\Delta(p)$
at primes $p\le O(\log^2{x})$ \cite{LLS}.
This statement alone does not offer any
indication of how to find $\Delta$ given a list of its initial character
values, but together with Prop.~\ref{converseprop} it suggests that a given
$\Delta$ might be
captured by the system \eqref{lpsystem} using test functions supported
up to $X\approx2\log\log|\Delta|$, provided that we are
allowed to take $V$ and $K$ arbitrarily large.
However, our numerical experiments so far, which were limited to at most
a few hundred integer variables, have not corroborated this speculation, even
allowing for larger values of $X$.

\subsection{Finding correlating characters}
\subsubsection{Lining up the initial primes}
\label{smallprimes}
In order to improve the efficiency of the brute force search, we
chose $q$ so as to line up the values of the prime sum for small $n$,
i.e.\ so that $\chi_{q\Delta}(p)=1$ for small primes $p$.  Of course there
is no guarantee that doing so is optimal, and indeed it is likely
that the best choices of $q$ of a given size adhere to this principle
only loosely, i.e.\ they may sacrifice a few small values of $p$ in order
to line up many more. However, if we have the resources to evaluate
the prime sum for a fixed number of $q$, regardless of size (a reasonable
assumption, since the only operation performed on $q$ itself
is reduction mod $p$), then it makes sense to line up the small primes
in attempt to skew the distribution of values in our favor.

To be more precise, consider an idealized form of the lower bound
\eqref{lowerbound2}
with $h$ a $\delta$-function and $g\equiv 1$.  Then for a prime power
$n=p^k$, the corresponding term of \eqref{lowerbound2} is
$2\chi_{q\Delta}(n)\Lambda(n)/\sqrt{n}$.
The expected value of this term, that is its average value
over all $q\in\mathcal{F}$, is easily seen to be
$0$ if $k$ is odd and $2\frac{p}{p+1}\frac{\Lambda(n)}{\sqrt{n}}$
if $k$ is even.\footnote{$\mathbb{E}(\chi_{q\Delta}(p^{2k}))
=\mathbb{E}(\chi_{q\Delta}(p^2)) = \frac{\phi(p^2)}{p^2-1}=\frac{p}{p+1}$,
where the second equality holds because $q\in\mathcal{F}$,
so that $q\not\equiv 0\pmod{p^2}$.}
Thus, if we force $q$ to satisfy $\chi_{q\Delta}(p)=1$, this introduces a
positive bias in the prime sum of
$$
\sum_{k=1}^\infty 2\frac{\Lambda(p^k)}{p^{k/2}}
\begin{cases}
\frac1{p+1}&\text{if }2\mid k,\\
1&\text{if }2\nmid k
\end{cases}
=\frac{2\log{p}}{p-1}\left(\sqrt{p}+\frac1{p+1}\right).
$$
However, it also comes with a price, in that we expect such a $q$
to be about $2(p+1)/p$ times larger than a fundamental
discriminant chosen randomly without regard to the value of $\chi_q(p)$.
Thus, our expected net improvement is
\begin{equation}
\label{primesumbias}
\frac{2\log{p}}{p-1}\left(\sqrt{p}+\frac1{p+1}\right)
-\log\frac{2(p+1)}p.
\end{equation}
(A similar argument applies to forcing $\chi_{q\Delta}(-1)=1$, from
which we expect a net improvement of $\frac{\pi}2-\log{2}$.)  It turns
out that \eqref{primesumbias} is positive for $p\le 251$ but negative
for larger primes.

\subsubsection{The shortest lattice vector problem}
It is plausible that there is a better strategy for finding good twists than 
a brute-force search, meaning a strategy that
can find the same quality twist as brute force but using much less
sampling. If one could be assured of finding
 $\gamma_1(q\Delta) \gg \sqrt{X}/\log|\Delta|$ in
a subset of $\{q\in\mathcal{F}:(q,\Delta)=1,|q|\le \exp{X}\}$ 
of size $\ll \exp(X^{\tau})$, $0\le \tau<1$, 
then one could improve the $2/3$ exponent 
to $\max\{\theta^*_{\min},\frac{2\tau}{2\tau+1}\}$, where
$\theta^*_{\min} := \inf_{\theta>0}\eta_{\infty}(\theta)$,
provided the subset can be determined easily.
An obvious candidate is the subset of smooth fundamental discriminants.
For example, one could search for a product of real primitive characters
$\chi_q=\chi_{q_1}\cdots\chi_{q_m}$, $|q_j|<Q$, $q_j\ne q_k$, that 
correlates strongly with $\chi_d$, so as to make  
\begin{equation}
\label{llltarget}
2\sum_{p<P} \frac{\log p}{\sqrt{p}} - 2\sum_{p<P} \frac{\chi_q(p)\chi_d(p)\log
p}{\sqrt{p}} + \log|q| 
\end{equation}
small, in the hope that it will lead to an unusually large
prime sum in the explicit formula.
The question of finding a good choice of $\chi_q$ 
can be framed in terms of finding a short
vector in the lattice generated by the rows of the following
$(n+m+1)\times(n+m+1)$ matrix, as we explain next. 
(This idea was applied in \cite{odlyzko-teriele} in 
the $t$-aspect to disprove the Mertens conjecture.)
\begin{displaymath}
\small
\begin{bmatrix}
i(d,p_1) &i(d,p_2) &\cdots & i(d,p_n) & 0 & 0 & \cdots&0 & 2^M \\
i(q_1,p_1) &i(q_1,p_2)& \cdots & i(q_1,p_n) & \lfloor 2^M\sqrt{\log|q_1|}\rfloor& 0 & \cdots&0 & 0 \\
i(q_2,p_1) &i(q_2,p_2)& \cdots & i(q_2,p_n) & 0 &
\lfloor2^M\sqrt{\log|q_2|}\rfloor & \cdots&0 & 0 \\
\vdots &\vdots &\ddots & \vdots & \vdots & \vdots & \ddots  &\vdots & \vdots  \\
i(q_m,p_1) &i(q_m,p_2)& \cdots & i(q_m,p_n) & 0 & 0 & \cdots & \lfloor2^M\sqrt{\log|q_m|} \rfloor& 0 \\
2w(p_1) & 0&\cdots & 0 &  0 & 0 & \cdots&0 &  0 \\
0 & 2w(p_2) & \cdots & 0 & 0& 0 & \cdots &  0 &0\\ 
\vdots &\vdots & \ddots& \vdots & \vdots & \vdots &\vdots & \vdots &  \vdots\\ 
0 & 0&\cdots & 2w(p_n) & 0 &0 & \cdots &0  & 0
\end{bmatrix}
\end{displaymath}
\begin{align*}
\small
w(p):= \left\lfloor \frac{2^{M+1}\sqrt{\log p}}{p^{1/4}}\right\rfloor, 
\quad i(q,p):=\frac12\left(1+\chi_{q^*}(p)\right)w(p),
\quad q^*=\begin{cases}
(-1)^{\frac{q-1}{2}}q &\textrm{if $q$ odd prime,}\\ 
q & \textrm{if $q\in\{-4,8,-8\}$.}
\end{cases}
\end{align*}
Here, $M$ is a large integer of our choice (in our application it was a random 
integer in $[75,150)$).
The weight $w(p)$ comes from \eqref{llltarget}, and
indicates that it is more important to correlate smaller primes. 
The weight $\sqrt{\log|q_j|}$ in the $(n+1)$st to $(n+m)$th columns indicates
that using $\chi_{q_j^*}$ will incur a penalty imposed according to
the explicit formula. 
The bottom $n$ rows indicate that the $k$th entry in each row, 
$1\le k\le n$, should be treated
modulo $2w(p_k)$. Here, it is helpful to note that 
$i(q_j,p_k)$ is either $0$ or $1$ times $w(p_k)$, and so working modulo
$2w(p_k)$ essentially means that only one multiple of each row is needed.
Therefore, a vector in the lattice with 
a non-zero $(n+m+1)$-st entry, can be written in the form
\begin{equation}
\label{charvector}
\left(y_1w(p_1),\ldots,y_nw(p_n),u_1\sqrt{\log|q_1|},\ldots,u_m\sqrt{\log|q_m|},2^M\right),
\end{equation}
where $y_k,u_j\in\{0,1\}$. 
The character generated by this vector
is $\chi_{q_\mathcal{J}}:=\prod_{j\in\mathcal{J}} \chi_{q_{j}^*}$,
where $\mathcal{J}:=\{1\le j\le m,u_{j}=1\}$, and 
it has discriminant $q_{\mathcal{J}} = \prod_{j\in\mathcal{J}} q_{j}^*$.
The $y_k$ are $0$ or $1$ according to whether 
$\chi_{\mathcal{J}}(p_k) = \chi_d(p_k)$
or not. Hence, in order for the vector \eqref{charvector} to be short, it means that
$$
\sum_{\chi_{q_\mathcal{J}}(p_k) \ne \chi_d(p_k)} w(p_k)^2 + \sum_{j\in\mathcal{J}}\left\lfloor
2^M\sqrt{\log|q_j|}\right\rfloor^2+2^{2M}
\approx 2^{2M}\left[4\sum_{\substack{\chi_{q_\mathcal{J}}(p_k) \ne\\ \chi_d(p_k)}} \frac{\log p_k}{\sqrt{p_k}} 
+ \sum_{j\in\mathcal{J}}\log|q_j|+1\right] 
$$
has to be small. This expression is essentially the same as \eqref{llltarget},
which we wish to minimize, but with $\chi_q = \chi_{q_\mathcal{J}}$ and
$q=q_{\mathcal{J}}$. Thus, it is seen that one can find a good choice of $\chi_q$ 
if one can find a short vector in the lattice.

Finding the shortest vector in a lattice 
is conjectured to be $NP$-hard in the $l^2$-norm and the corresponding decision
problem is conjectured to be $NP$-complete (see \cite{ajtai}).
However, one can find relatively short vectors 
in polynomial time using the 
LLL algorithm of Lenstra, Lenstra, and Lov\'as \cite{lll}, which 
 produces a basis that is nearly orthogonal.
The LLL algorithm was first used to factor a primitive
univariate polynomial in polynomial time. It does not necessarily find the
shortest vector, and it usually does not, but it can find relatively
short vectors quickly.

We applied LLL to our lattice with $P$ and $Q$ ranging between $100$ to 
over $1000$. While it did yield above-average choices of $\chi_q$, such as $q=-73147$,
our best-performing twists ultimately came from the brute-force
approach described in \S\ref{smallprimes}.

\subsection{More general twists}
If $\pi$ is a cuspidal automorphic representation of $GL_r(\A_\Q)$
with conductor $q$ relatively prime to $\Delta$, then the twist
$\pi\otimes\chi_\Delta$ has conductor $q|\Delta|^r$. Assuming GRH for
the associated $L$-function $L(s,\pi\otimes\chi_\Delta)$, we get a lower
bound for $\log|\Delta|$ via the explicit formula.  Thus, the idea of
using twists as in \S\ref{twistsection} admits a vast
generalization.

For any natural family of twists, one can expect a more general version of
Prop.~\ref{optimalprop} to hold, i.e.\ for each $X>0$ there will be some
optimal choice of input data $(\pi,g)$, where $\pi$ is an element of the
family and $g\in\mathcal{C}(X)$ is a test function to use in the explicit
formula. For instance, considering the family of quadratic twists as in
\S\ref{twistsection}, it is easy to see that the right-hand side
of \eqref{lowerbound2} is bounded above by $O_X(1)-\log|q|$, uniformly
for $g\in\mathcal{C}(X)$.  Thus, only finitely many $q$ are relevant,
so it follows from Prop.~\ref{optimalprop} that there is a pair
$(q,g)\in\mathcal{F}\times\mathcal{C}(X)$
which maximizes the lower bound \eqref{lowerbound2}.

Similarly, one can expect an analogue of Prop.~\ref{runningtimeprop} to
hold for any given family. It would be of interest to study other families
to see which yield the best performance.
We conclude by listing a few families of $L$-functions that would make
good candidates for future investigations.
\begin{itemize}
\item \emph{Elliptic curve $L$-functions}. Can one make use of the BSD
conjecture and the existence of high-order zeros at the central point
to force zero repulsion?
\item \emph{Dedekind $\zeta$-functions}. Can one make use of the
existence of towers of number fields of bounded root discriminant?
\item \emph{Rankin--Selberg products}. Can one make use of the algebraic
structure of the coefficients of $L$-functions to find correlating
twists quickly?
\end{itemize}

\appendix
\section{Proofs}
\subsection{Proof of Prop.~\ref{optimalprop}}
Let $g_n\in\mathcal{C}(X)$, $n=1,2,3,\ldots$, be a maximizing sequence
for $l(g)$, with corresponding
cosine transforms $h_n$. Since each $h_n$ is non-negative, we have
$|g_n(x)|\le g_n(0)=1$. Therefore, for $j\in\Z_{\ge0}$,
\begin{equation}
\label{hnjbound}
\bigl|h_n^{(2j)}(0)\bigr|=\left|2(-1)^j\int_0^Xx^{2j}g_n(x)\,dx\right|
\le\frac{X^{2j+1}}{j+\frac12},
\end{equation}
so that $h_n^{(2j)}(0)$ varies within a compact set for each fixed $j$.
Applying Cantor's diagonal argument, we may assume without loss of
generality that the sequence $\bigl\{h_n^{(2j)}(0)\bigr\}_{n=1}^\infty$
converges for every $j$.
Put $c_j=\lim_{n\to\infty}h_n^{(2j)}(0)$ and
$h_\infty(t)=\sum_{j=0}^{\infty}\frac{c_j}{(2j)!}t^{2j}$.
Then from \eqref{hnjbound} it follows that $h_\infty$ is an entire function
and $h_n(t)$ converges uniformly to $h_\infty(t)$ on compact subsets of $\C$.
In particular, $h_\infty(t)\ge0$ for all $t\in\R$.

Next, for any $g\in\mathcal{C}(X)$ with cosine transform $h$, we have
\begin{align*}
\log(8\pi e^\gamma)
&-\int_0^\infty\frac{1-g(x)}{2\sinh(x/2)}\,dx
+\chi_\Delta(-1)\int_0^\infty\frac{g(x)}{2\cosh(x/2)}\,dx\\
&=-\frac1{\pi}\int_\R\Re\frac{\Gamma_\R'}{\Gamma_\R}\!\left(
\frac12+a+it\right)h(t)\,dt,
\end{align*}
where $\Gamma_\R(s)=\pi^{-s/2}\Gamma(s/2)$ and $a\in\{0,1\}$ is such
that $(-1)^a=\chi_\Delta(-1)$.
By Stirling's formula we have
$\Re\frac{\Gamma_\R'}{\Gamma_\R}\!\left(\frac12+a+it\right)
=\frac12\log(1+|t|)+O(1)$,
so that
$$
\frac1{\pi}\int_\R
\Re\frac{\Gamma_\R'}{\Gamma_\R}\!\left(\frac12+a+it\right)h(t)\,dt
=\frac1{\pi}\int_0^\infty\log(1+t)h(t)\,dt+O(1).
$$
Moreover, since $|g(x)|\le 1$ for all $x$, we have
$$
2\sum_{n=1}^{\infty}\frac{\Lambda(n)\chi_\Delta(n)}{\sqrt{n}}g(\log{n})
\ll 1,
$$
where the implied constant depends only on $X$.

Returning to our construction, since $g_n$ is a maximizing sequence,
we may assume without loss of generality that $l(g_n)$ is
bounded below. Together with the above observations, we thus have that
$\frac1{\pi}\int_0^\infty\log(1+t)h_n(t)\,dt\le C$ for some constant $C$.
Hence, for any $T>0$, we have
$$
\frac1{\pi}\int_0^T\log(1+t)h_\infty(t)\,dt
=\lim_{n\to\infty}\frac1{\pi}\int_0^T\log(1+t)h_n(t)\,dt
\le C.
$$
Since $T$ is arbitrary and $\log(1+t)h_\infty(t)$ is non-negative, we see
that
\begin{equation}
\label{hsintbound}
\frac1{\pi}\int_0^\infty\log(1+t)h_\infty(t)\,dt\le C.
\end{equation}
In particular, $h_\infty\in L^1([0,\infty))$,
so its cosine transform
$g_\infty(x)=\frac1{\pi}\int_0^{\infty}h_\infty(t)\cos(xt)\,dt$
is well-defined and continuous.
Moreover, by \eqref{hsintbound} we have
$$
\frac1{\pi}\int_T^\infty h_\infty(t)\,dt
\le\frac1{\pi}\int_0^\infty\frac{\log(1+t)}{\log(1+T)}h_\infty(t)\,dt
\le\frac{C}{\log(1+T)},
$$
and similarly
$\frac1{\pi}\int_T^\infty h_n(t)\,dt\le\frac{C}{\log(1+T)}$ for all $T>0$.
Therefore,
$$
\left|g_n(x)-\frac1{\pi}\int_0^T h_n(t)\cos(xt)\,dt\right|
\le\frac{C}{\log(1+T)}
$$
and
$$
\left|g_\infty(x)-\frac1{\pi}\int_0^T h_\infty(t)\cos(xt)\,dt\right|
\le\frac{C}{\log(1+T)}.
$$

Now, let $\varepsilon>0$ be given, and choose $T>0$ large enough that
$\frac{C}{\log(1+T)}\le\frac{\varepsilon}3$.
Further, let $N\in\Z_{\ge0}$ be such that $n>N$ implies that
$|h_n(t)-h_\infty(t)|<\frac{\pi}{3T}\varepsilon$ for all $t\in[0,T]$.
Then the above inequalities yield
$$
\bigl|g_n(x)-g_\infty(x)\bigr|\le\frac{2C}{\log(1+T)}+
\left|\frac1{\pi}\int_0^T\bigl(h_n(t)-h_\infty(t)\bigr)\cos(xt)\,dt\right|
<\varepsilon
$$
for all $n>N$ and $x\ge0$. Thus,
$g_n(x)$ converges uniformly to $g_\infty(x)$.
In particular, $g_\infty$ is supported on $[0,X]$ and satisfies
$g_\infty(0)=1$, so it is an element of $\mathcal{C}(X)$.

Finally, let $\delta>0$ be given.
By Stirling's formula, there is a number $T_0>0$ such that
$\Re\frac{\Gamma_\R'}{\Gamma_\R}\!\left(\frac12+a+it\right)\ge0$
whenever $|t|\ge T_0$.
Moreover, it follows from \eqref{hsintbound} that
the function
$\Re\frac{\Gamma_\R'}{\Gamma_\R}\!\left(\frac12+a+it\right)h_\infty(t)$
is absolutely integrable, so 
there exists $T\ge T_0$ such that
$$
0\le\frac1{\pi}\int_{\R\setminus[-T,T]}
\Re\frac{\Gamma_\R'}{\Gamma_\R}\!\left(\frac12+a+it\right)h_\infty(t)\,dt
<\delta.
$$
Therefore,
\begin{align*}
l(g_\infty)+\delta&>
2\sum_{n=1}^{\infty}\frac{\Lambda(n)\chi_\Delta(n)}{\sqrt{n}}g_\infty(\log{n})
-\frac1{\pi}\int_{-T}^T\Re\frac{\Gamma_\R'}{\Gamma_\R}\!\left(
\frac12+a+it\right)h_\infty(t)\,dt\\
&=\lim_{m\to\infty}\left(
2\sum_{n=1}^{\infty}\frac{\Lambda(n)\chi_\Delta(n)}{\sqrt{n}}g_m(\log{n})
-\frac1{\pi}\int_{-T}^T\Re\frac{\Gamma_\R'}{\Gamma_\R}\!\left(
\frac12+a+it\right)h_m(t)\,dt\right)\\
&\ge\lim_{m\to\infty}\left(
2\sum_{n=1}^{\infty}\frac{\Lambda(n)\chi_\Delta(n)}{\sqrt{n}}g_m(\log{n})
-\frac1{\pi}\int_\R\Re\frac{\Gamma_\R'}{\Gamma_\R}\!\left(
\frac12+a+it\right)h_m(t)\,dt\right)\\
&=\sup_{g\in\mathcal{C}(X)}l(g)\ge l(g_\infty).
\end{align*}
Since $\delta$ is arbitrary, we have
$l(g_\infty)=\sup_{g\in\mathcal{C}(X)}l(g)$.
\qed

\subsection{Proof of Prop.~\ref{testfunctionprop}}
We begin with some lemmas.
\begin{lemma}
\label{gnuXlemma}
For $\nu>0$, define
$$
f_{\nu}(x)=\begin{cases}
\bigl(1-x^2\bigr)^\nu&\text{if }|x|<1,\\
0&\text{otherwise},
\end{cases}
$$
$$
g_{\nu,X}(x)=\frac{\Gamma(\frac32+2\nu)}{\sqrt{\pi}\Gamma(1+2\nu)}
\int_{\R}f_{\nu}(y)f_{\nu}\!\left(\frac{2x}X-y\right)dy
\quad\text{for }x\ge0,
$$
and
$h_{\nu,X}(t)=2\int_0^\infty g_{\nu,X}(x)\cos(tx)\,dx$.
Then $g_{\nu,X}(0)=1$, $g_{\nu,X}$ is supported on $[0,X]$,
$h_{\nu,X}$ is non-negative and satisfies
$$
h_{\nu,X}(t)\ll_\varepsilon\nu^{-1/2}e^{-2\nu}X\left|\frac{4\nu}{Xt}\right|^{2\nu+2}
\quad\text{uniformly for }\left|\frac{Xt}4\right|\ge\nu\ge\varepsilon,
$$
for any fixed $\varepsilon>0$.
\end{lemma}
\begin{proof}
Using the Poisson representation for the $J$-Bessel function,
we derive
$$
\int_{\R}f_{\nu}(x)e^{itx}\,dx=2\Gamma(1+\nu)j_\nu(|t|)
\left|\frac2t\right|^{\nu},
$$
where $j_\nu(u)=\sqrt{\frac{\pi}{2u}}J_{\nu+1/2}(u)$ is the spherical
Bessel function with parameter $\nu$.  From this and the limit
$j_\nu(u)(\frac2u)^\nu\to\frac{\sqrt{\pi}}{2\Gamma(\frac32+\nu)}$
as $u\to 0^+$, we derive
$$
\int_{\R}f_\nu(x)^2\,dx=\int_{\R}f_{2\nu}(x)\,dx=
\frac{\sqrt{\pi}\Gamma(1+2\nu)}{\Gamma(\frac32+2\nu)},
$$
so that $g_{\nu,X}(0)=1$.  Therefore, we have
$$
h_{\nu,X}(t)=\frac{2X}{\sqrt\pi}
\frac{\Gamma(\frac32+2\nu)\Gamma(1+\nu)^2}{\Gamma(1+2\nu)}
j_\nu\!\left(\frac{X|t|}2\right)^2\left|\frac4{Xt}\right|^{2\nu}.
$$
It follows from \cite[Thm.~2]{krasikov} that the function $uj_\nu(u)$
is bounded in the
region $\{(\nu,u):u\ge2\nu\ge0\}$.  This combined with Stirling's formula
gives the estimate.
\end{proof}

\begin{lemma}
\label{Galphalemma}
For $\alpha>0$, set
$$
H_\alpha(t)=\frac{\alpha}2\sinc^2\!\left(\frac{\alpha t}2\right)
+\frac{\alpha}{4}\left[
\sinc^2\!\left(\frac{\alpha t-\pi}2\right)
+\sinc^2\!\left(\frac{\alpha t+\pi}2\right)\right]
$$
and $G_\alpha(x)=\frac1{\pi}\int_0^\infty H_\alpha(t)\cos(tx)\,dt$ for
$x\ge0$.  Then $G_\alpha$ is supported on $[0,\alpha]$, $G_\alpha(0)=1$,
and $H_\alpha(t)\ge\frac{2\alpha}{\max((\alpha t)^2,\pi^2/3)}$ for $t\in\R$.
\end{lemma}
\begin{proof}
A straightforward calculation gives
\begin{equation}
\label{Galpha}
G_\alpha(x)=\max\!\left(0,1-\frac{x}{\alpha}\right)
\cos^2\!\left(\frac{\pi{x}}{2\alpha}\right),
\end{equation}
which yields the stated properties of $G_\alpha$.
As for $H_\alpha$, by rescaling, it suffices to prove the bound for
$\alpha=1$. A calculation shows that
\begin{align*}
H_1(t)=\frac2{t^2}+2\left(\frac{\pi\cos(t/2)}{t(t^2-\pi^2)}\right)^2
(3t^2-\pi^2),
\end{align*}
so that $H_1(t)\ge 2t^{-2}$ for $|t|\ge\pi/\sqrt{3}$.
On the other hand,
graphing the function verifies that $H_1(t)\ge6/\pi^2$ for 
$|t|\le\pi/\sqrt{3}$.
\end{proof}

Now, turning to Prop.~\ref{testfunctionprop},
first note that $|q\Delta|\ge 3$.
We take $h(t)=h_{\nu,X}(t)$,
where $\nu=\frac{\delta X}4\ge\frac12\log\log{3}>0$.
Applying Lemma~\ref{gnuXlemma} with $\varepsilon=\frac12\log\log{3}$ and
Lemma~\ref{Galphalemma} with $\alpha=2\log\log|q\Delta|$,
for $t\ge\delta$ we have
\begin{align*}
h(t)&\ll\nu^{-1/2}e^{-2\nu}X\left(\frac{4\nu}{Xt}\right)^{2\nu+2}
\le X(\delta X)^{-1/2}e^{-\delta X/2}
\alpha^{-1}\max\left((\alpha\delta)^2,\frac{\pi^2}3\right)H_\alpha(t)\\
&\le\frac{e^{-A}X}{(2\log\log|q\Delta|)^{3/2}}
\frac{\max((2\delta\log\log|q\Delta|)^2,\pi^2/3)}{\log|q\Delta|}
H_\alpha(t).
\end{align*}
Since $\gamma_j(q\Delta)\ge\delta$ for every $j$, this yields
$$
\sum_{j=1}^{\infty}h(\gamma_j(q\Delta))
\ll\frac{e^{-A}X}{(\log\log|q\Delta|)^{3/2}}
\frac{\max((2\delta\log\log|q\Delta|)^2,\pi^2/3)}{\log|q\Delta|}
\sum_{j=1}^{\infty}H_\alpha(\gamma_j(q\Delta)).
$$

We estimate the latter sum by plugging back into the explicit formula
\eqref{explicitformula}, with $\Delta$ replaced by $q\Delta$.
A calculation with the prime number theorem using \eqref{Galpha}
shows that
$$
\sum_{n=1}^{\infty}\frac{\Lambda(n)}{\sqrt{n}}G_\alpha(\log{n})
\ll\frac{e^{\alpha/2}}{\alpha^3},
$$
and it is not hard to see that
$$
\log(8\pi e^\gamma)
-\int_0^\infty\frac{1-G_\alpha(x)}{2\sinh(x/2)}\,dx
+\chi_{q\Delta}(-1)\int_0^\infty\frac{G_\alpha(x)}{2\cosh(x/2)}\,dx
=O(1)
$$
uniformly for $\alpha\ge2\log\log{3}$.
Thus, $\sum_{j=1}^\infty H_\alpha(\gamma_j(q\Delta))\ll \log|q\Delta|$.

Finally, by \cite[Thm.~11]{omar}, under GRH we have
$\delta\ll 1/\log\log|q\Delta|$.  This yields \eqref{zerosumbound}.
\qed

\subsection{Proof of Prop.~\ref{runningtimeprop}}
We may assume without loss of generality that $N$ is odd,
not a perfect square, and satisfies $N\ge\exp(\exp(C^{2/3}))$, where
$C>0$ is the implied constant in \eqref{zerosumbound}.\footnote{If $C$
is at all large then this rules out every $N$ of practical size; we could
deal with this instead by increasing $A$ in Prop.~\ref{testfunctionprop}
by a constant, but as we are only interested in the theoretical result,
we make this assumption for convenience.} Thus,
if we set $d=(-1)^{\frac{N-1}2}N$ then $d=\Delta\ell^2$ for some
$1\ne\Delta\in\mathcal{F}$ and $\ell\in\Z_{>0}$.

Let $Q\ge3$ be an integer parameter to be specified later.
Set $\nu=\frac12\log\log{N}+\frac1{12\log\log{N}}$,
$X=4\nu\log{Q}$, and let $T$ be an integer in the
interval $[e^X-1,e^X+1)$. (Note that to find such a $T$, it suffices to
compute $e^X$ to within $\pm\frac12$.)
We let $q$ run through all elements of $\mathcal{F}$ with
$|q|\le Q$ and evaluate the lower bound \eqref{lowerbound2} using the
test function $g=g_{\nu,X}$, in the notation of Lemma~\ref{gnuXlemma}.

Since $g_{\nu,X}(\log{n})=0$ for $n>T$, it is enough to consider the
terms of the sum for $n\le T$.  As described in \S\ref{fdsection},
since $\Delta$ is unknown to us, we compute $\chi_d(n)$ in place of
$\chi_\Delta(n)$. If for any prime value of $n$ we find a zero value
of $\chi_d(n)$, we check to see if $n^2|N$ and exit with this square
factor if so; otherwise $\chi_\Delta(n)=\chi_d(n)$.  In particular, while
computing \eqref{lowerbound2} for $q=1$, we evaluate $\chi_\Delta(n)$
for all primes $n\le T$.  If $(T+1)^3>N$ then this alone yields enough
information to determine whether $N$ is squarefree. Hence, we may
assume without loss of generality that $X\le\frac13\log{N}$, so that
$\log{Q}\le\frac{\log{N}}{12\nu}$.

Note that if we set $A=2\nu-\log\log|q\Delta|$ and $\delta=\frac1{\log{Q}}$
then $g_{\nu,X}$ is precisely the test function exhibited in the proof of
Prop.~\ref{testfunctionprop}. Using the bound
$|q\Delta|\le QN\le N^{1+\frac1{12\nu}}$,
we derive the inequality
$1-e^{-A}>\frac1{72}(\log\log{N})^{-2}>0$.
Thus, Prop.~\ref{testfunctionprop} shows that if
$|\Delta|=N$ (which holds when $N$ is squarefree)
and $\gamma_1(q\Delta)\ge\frac1{\log{Q}}$ then
$$
\sum_{j=1}^{\infty}h_{\nu,X}(\gamma_j(q\Delta))
<\left(1-\frac1{72(\log\log{N})^2}\right)\frac{CX}{(\log\log(qN))^{3/2}}
\le\left(1-\frac1{72(\log\log{N})^2}\right)X.
$$
Therefore, if we evaluate \eqref{lowerbound2} to within
$\pm\frac{X}{72(\log\log{N})^2}$, we will have proven that
$|\Delta|>Ne^{-2X}$, so that $\ell\le T$.  Having already determined all
prime factors of $N$ up to $T$, we will thus have found a proof that $N$
is squarefree.

Now, since the value of $\theta^*$ is unknown to us, we cannot say in
advance what value of $Q$ will suffice. In our algorithm, we therefore
apply the above procedure iteratively with $Q=2^k$ for $k=2,3,4,\ldots$
until we find either a square factor or a proof that $N$ is squarefree.
As noted above, the algorithm must eventually terminate.  If it turns out
that $\theta^*=1$ or if $N$ has a square factor then the algorithm becomes
a rather inefficient version of trial division, which nevertheless runs
in polynomial time in $N$; in particular, the $O(\exp[(\log{N})^{1+o(1)}])$
running time estimate holds. Henceforth we will assume that $\theta^*<1$
and that the input $N$ is squarefree.

Fix $\varepsilon\in(0,1-\theta^*)$.  From the definition of $\theta^*$
it follows that $\eta_\infty(\theta^*+\varepsilon)\le\theta^*$.
Thus, there exists $N_0(\varepsilon)\in\Z_{>0}$ such that
$\eta_\Delta(\theta^*+\varepsilon)<\theta^*+\varepsilon$
whenever $|\Delta|=N\ge N_0(\varepsilon)$. Let
us assume that $N\ge N_0(\varepsilon)$. Then once
$\frac{\log\log{Q}}{\log\log{N}}\ge\theta^*+\varepsilon$, there
must be a $q$ with $(q,\Delta)=1$ and $|q|\le Q$ such that
$\gamma_1(q\Delta)>\frac1{\log{Q}}$.

It is straightforward to see that all of the floating point
operations required to compute \eqref{lowerbound2} for every $|q|\le
Q$ to the precision described above may be carried out in time
$O(Q^{1+4\nu}\log^c{N})$ for some $c>0$.  Since we choose values of $Q$
from a geometric progression, the total running time is dominated by that
of the final iteration. In the worst case, it might be that the smallest
$Q$ for which $\frac{\log\log{Q}}{\log\log{N}}\ge\theta^*+\varepsilon$
is $2^k+1$ for some $k$, and thus our final choice of $Q=2^{k+1}$ would
be too large by roughly a factor of $2$. Thus,
$\log{Q}\le(\log{N})^{\theta^*+\varepsilon}+\log{2}$,
so that $Q^{1+4\nu}\log^c{N}\ll
\exp\bigl[(\log{N})^{\theta^*+\varepsilon}(1+4\nu)\bigr]
(\log{N})^{c+\log{4}}$. Since $1+4\nu\ll\log\log{N}$
and $\varepsilon$ may be chosen arbitrarily small (assuming only
that $N\ge N_0(\varepsilon)$), the running time is thus
$O\bigl(\exp\bigl[(\log{N})^{\theta^*+o(1)}\bigr]\bigr)$, as required.
\qed

\subsection{Proof of Prop.~\ref{uspgapprop}}
\begin{lemma}
\label{uspgapbounds}
For each $N\ge 1$, and each $s\in(0,\pi)$, we have
\begin{equation}
\label{uspgapineq}
1\le\frac{\mathbb{P}_{USp(2N)}\bigl(\theta_1>s\bigr)}{\cos(s/2)^{N(2N+1)}}
\le\frac12\left(1+\frac{\sin(s/2)}{\sqrt2}\right)^{2N+1}
+\frac12\left(1-\frac{\sin(s/2)}{\sqrt2}\right)^{2N+1}
\le\exp\!\left(\frac{Ns}{\sqrt2}\right).
\end{equation}
In particular, we have
\begin{equation}
\label{uspgapasympt}
\log \mathbb{P}_{USp(2N)}\bigl(\theta_1>s\bigr) =
\bigl[2+O\bigl((Ns)^{-1}\bigr)\bigr]N^2\log\cos(s/2),
\end{equation}
uniformly for $s\in(0,\pi)$.
\end{lemma}
\begin{proof}
The Weyl integration formula on $USp(2N)$ gives
\begin{equation}
\begin{split}
\mathbb{P}_{USp(2N)}\bigl(\theta_1>s\bigr) &= \int_{USp(2N)} 1_{\theta_1>s}\,d\mathbb{P}_{USp(2N)}\\ 
&= \frac{2^{N^2}}{\pi^N\,
N!}\int_{(s,\pi]^N} \prod_{1\le j< k\le N} (\cos \phi_k - \cos \phi_j)^2
\prod_{1\le j \le N} (\sin^2\phi_j) \,d\phi_1\cdots d\phi_N.\nonumber
\end{split}
\end{equation}
We proceed along the lines of the proof of \cite[proposition
6.10.1]{ks}.\footnote{$\mathbb{P}_{G(N)}\bigl(\theta_1>s\bigr)$ is $\eigen(0,s,G(N))$
in the notation of \cite{ks}.}  
Applying the change of variable $\phi_j = 2 \tau_j$, the trig identities 
$\sin (2\tau_j)  = 2\sin \tau_j \cos \tau_j$ and $\cos(2\tau_j) = 2\cos^2\tau_j - 1$, 
and, last, the substitution $w_j = \cos^2 \tau_j$, we obtain 
$\mathbb{P}_{USp(2N)}\bigl(\theta_1>s\bigr) = I_N(\cos^2(s/2))$, where
$$
I_N(\lambda) := 
\frac{C(N)}{N!}\int_{[0,\lambda)^N} \prod_{1\le j< k\le N} (w_k - w_j)^2
\prod_{1\le j \le N} \sqrt{w_j(1-w_j)} \,dw_1\cdots dw_N,
$$
and $C(N):=2^{2N^2+N}/\pi^N$. The
change of variable $w_j = \lambda x_j$ thus yields
\begin{equation}
\label{gapprobintegral}
I_N(\lambda) = \lambda^{N^2+N/2} 
\frac{C(N)}{N!}\int_{[0,1)^N} \prod_{1\le j< k\le N} (x_k - x_j)^2
\prod_{1\le j \le N} \sqrt{x_j(1-\lambda x_j)} \,dx_1\cdots dx_N.
\end{equation}
Since $\sqrt{1-\lambda x_j} \ge \sqrt{1-x_j}$ for $0\le \lambda \le 1$, 
we have $I_N(\lambda) \ge \lambda^{N^2+N/2} I_N(1)$. The lower bound follows 
on observing that $I_N(1) = \mathbb{P}_{USp(2N)}\bigl(\theta_1>0\bigr) = 1$.

By comparing the joint probability density functions 
of eigenphases in $USp(2N)$ and $SO(2N)$, 
one easily obtains the rough upper bound  
$\mathbb{P}_{USp(2N)}\bigl(\theta_1>s\bigr) \le 
4^N \mathbb{P}_{SO(2N)}\bigl(\theta_1>s\bigr)$. 
Combined with the estimate
$\mathbb{P}_{SO(2N)}\bigl(\theta_1>s\bigr) \le \cos(s/2)^{2N^2-N}$
from \cite[proposition 6.10.1]{ks} and the lower bound $I_N(\lambda) \ge
\lambda^{N^2+N/2}$, this yields the asymptotic \eqref{uspgapasympt} in the range 
$\beta \in (1,2)$. To derive the upper bound, and consequently the asymptotic, 
in the full range, we apply the first-order inequality
\begin{align*}
\sqrt{1-\lambda{x_j}}&=\sqrt{1-x_j}\sqrt{1+\frac{(1-\lambda)x_j}{1-x_j}}
\le\sqrt{1-x_j}\left(1+\frac{(1-\lambda)x_j}{2(1-x_j)}\right)\\
&=\frac{1+\lambda}2\sqrt{1-x_j}
\left(1+\frac{1-\lambda}{1+\lambda}\frac1{1-x_j}\right),
\end{align*}
to get
\begin{align*}
\prod_{j=1}^N\sqrt{1-\lambda{x_j}}&\le
\left(\frac{1+\lambda}2\right)^N
\prod_{j=1}^N\sqrt{1-x_j}\cdot\sum_{J\subset\{1,\ldots,N\}}
\left(\frac{1-\lambda}{1+\lambda}\right)^{\#J}
\prod_{j\in J}(1-x_j)^{-1}\\
&=\left(\frac{1+\lambda}2\right)^N
\prod_{j=1}^N\frac1{\sqrt{1-x_j}}\cdot\sum_{J\subset\{1,\ldots,N\}}
\left(\frac{1-\lambda}{1+\lambda}\right)^{\#J}
\prod_{j\in\{1,\ldots,N\}\setminus J}(1-x_j),
\end{align*}
where $J$ runs through all subsets of $\{1,\ldots,N\}$.
We insert this into \eqref{gapprobintegral} and permute the variables so that
$\{1,\ldots,N\}\setminus J$ is mapped to $\{1,\ldots,N-\#J\}$.
Collecting the terms with a common value of $\#J$, we have
\begin{align*}
\frac{I_N(\lambda)}{\lambda^{N^2+N/2}}\le
\frac{C(N)}{N!}\left(\frac{1+\lambda}2\right)^N
&\sum_{m=0}^N{N\choose m}\left(\frac{1-\lambda}{1+\lambda}\right)^m\\
&\cdot\int_{[0,1)^N}\prod_{1\le j<k\le N}(x_k-x_j)^2
\prod_{j=1}^N\sqrt{\frac{x_j}{1-x_j}}
\prod_{j=1}^{N-m}(1-x_j)\,dx_1\cdots dx_N.
\end{align*}
By Aomoto's formula \cite[(17.1.6)]{mehta}, 
the integral may be written in the form
$$
K_N\prod_{j=1}^{N-m}\frac{\frac12+N-j}{1+2N-j}
=K_N\frac{\Gamma(N+\frac12)}{\Gamma(m+\frac12)}
\frac{\Gamma(1+N+m)}{\Gamma(1+2N)},
$$
where $K_N$ (a Selberg integral, see~\cite[(17.1.3)]{mehta}) 
is independent of $m$.  When $m=0$, the
integral is easily recognized as the one occurring in \eqref{gapprobintegral} with
$\lambda=1$, so that
$$
1=I_N(1)=\frac{C(N)K_N}{N!}
\frac{\Gamma(N+\frac12)}{\Gamma(\frac12)}
\frac{\Gamma(1+N)}{\Gamma(1+2N)}.
$$
Solving for $C(N)K_N/N!$ and substituting back into the above, we have
\begin{align*}
\frac{I_N(\lambda)}{\lambda^{N^2+N/2}}&\le
\left(\frac{1+\lambda}2\right)^N
\sum_{m=0}^N{N\choose m}\left(\frac{1-\lambda}{1+\lambda}\right)^m
\frac{\Gamma(\frac12)}{\Gamma(m+\frac12)}
\frac{\Gamma(1+N+m)}{\Gamma(1+N)}\\
&=\left(\frac{1+\lambda}2\right)^N
\sum_{m=0}^N{{N+m}\choose{2m}}\left(4\frac{1-\lambda}{1+\lambda}\right)^m.
\end{align*}
The last sum is known as a Morgan-Voyce polynomial; it is closely
related to the Chebyshev polynomials, and may
be evaluated in closed form.
Precisely, if $t$ is such that
$\cosh{t}=\sqrt{\frac2{1+\lambda}}$, then it follows from
\cite[(11b)]{swamy} that the last line is
$$
\frac{\cosh((2N+1)t)}{\cosh^{2N+1}t}
=\frac12\left(1+\sqrt{\frac{1-\lambda}2}\right)^{2N+1}
+\frac12\left(1-\sqrt{\frac{1-\lambda}2}\right)^{2N+1}.
$$
The upper bound follows on putting $\lambda=\cos^2(s/2)$ and noting that
$$
\frac12\left(1+\frac{\sin(s/2)}{\sqrt2}\right)^{2N+1}
+\frac12\left(1-\frac{\sin(s/2)}{\sqrt2}\right)^{2N+1}
\le\exp\!\left(\frac{Ns}{\sqrt2}\right).
$$
Combining this with the lower bound and the inequality
$|\log\cos(s/2)|\ge s^2/8$, we get the estimate
$$
\log\mathbb{P}_{USp(2N)}(\theta_1>s)=N(2N+1)\log\cos(s/2)+O(Ns)
=\bigl[2+O\bigl((Ns)^{-1}\bigr)\bigr]N^2\log\cos(s/2).
$$
\end{proof}
Turning to the proof of Prop.~\ref{uspgapprop},
by definition of 
$\mathbb{P}_{USp(2N)} (\max_{1\le m\le M} \theta_1(m) \le s )$,
we have, for each $s\in[0,\pi]$,
\begin{equation}
\label{sampleindependence}
\mathbb{P}_{USp(2N)} \bigl(\max_{1\le m\le M} \theta_1(m) \le s \bigr) =
\mathbb{P}_{USp(2N)} \bigl(\theta_1 \le s \bigr)^M.
\end{equation}
Suppose $\varepsilon<4$ and $s \in
[s^-_{\varepsilon,\beta}(N),s^+_{\varepsilon,\beta}(N)]$.
Then $s \to 0$ as $N\to\infty$ (since $\beta<2$ by assumption). Further,
by Lemma~\ref{uspgapbounds} we have $\mathbb{P}_{USp(2N)}(\theta_1>s) \to
0$ as $N\to\infty$ (since $\beta>0$ by assumption). 
Using \eqref{sampleindependence}, and Lemma~\ref{uspgapbounds} again, yields
\begin{equation}
\label{uspasymptderivation}
\begin{split}
\log \mathbb{P}_{USp(2N)} \bigl(\max_{1\le m\le M} \theta_1(m) \le s \bigr)
& = M \log \bigl(1 - \mathbb{P}_{USp(2N)}(\theta_1>s)\bigr)\\
& = -M \, \mathbb{P}_{USp(2N)}\bigl(\theta_1>s\bigr)(1+o(1)) \\
& = -\exp\bigl((2N)^{\beta}+[2+O((Ns)^{-1})]N^2\log\cos(s/2) +o(1)\bigr)\\
& = -\exp\bigl((2N)^{\beta}-(1+o(1))(Ns/2)^2+o(1)\bigr), 
\end{split}
\end{equation}
where we used that $\log\cos(s/2)\sim -s^2/8$ in the last line. 
So if $s=s^-_{\varepsilon,\beta}(N) = (4-\varepsilon)(2N)^{\beta/2-1}$, then
\begin{equation}
\begin{split}
\log \mathbb{P}_{USp(2N)} \bigl(\max_{1\le m\le M} \theta_1(m) \le
s\bigr) =
-\exp\bigl((2N)^{\beta} - (1+o(1))(1-\varepsilon/4)^2(2N)^{\beta}+o(1)\bigr)\\
= -\exp\bigl((1+o(1))(2N)^{\beta}(8\varepsilon-\varepsilon^2)/16+o(1)\bigr) \to -\infty,
\end{split}
\end{equation}
as $N\to\infty$, provided that $0<\varepsilon< 4$, which ensures that 
$(2N)^{\beta}(8\varepsilon-\varepsilon^2)/16 \to\infty$. 
If $\varepsilon \ge 4$ then clearly the result still holds.
Therefore, for each $\varepsilon>0$, we have 
$\mathbb{P}_{USp(2N)}\bigl(s^-_{\varepsilon,\beta}(N)
<\max_{1\le m\le M} \theta_1(m)\bigr)\to 1$, as claimed.

Similarly, if $s=s^+_{\varepsilon,\beta}(N) =(4+\varepsilon)(2N)^{\beta/2-1}$,
then
\begin{align*}
\log \mathbb{P}_{USp(2N)} \bigl(\max_{1\le m\le M} \theta_1(m) \le
s\bigr) &=
-\exp\bigl((2N)^{\beta} - (1+o(1))(1+\varepsilon/4)^2(2N)^{\beta}+o(1)\bigr)\\
&= -\exp\bigl(-(1+o(1))(2N)^{\beta}(8\varepsilon+\varepsilon^2)/16+o(1)\bigr),
\end{align*}
which tends to $0$ with $N$. Therefore, $\mathbb{P}_{USp(2N)}
\bigl(\max_{1\le m\le M} \theta_1(m) \le s^+_{\varepsilon,\beta}(N)\bigr)\to 1$.
\qed

\subsection{Proof of Prop.~\ref{ungapprop}}
We make use of the main results in \cite{di} and \cite{krasovsky}, 
formulated in Lemma~\ref{diLemma} here.
\begin{lemma}
\label{diLemma}
For $\delta > 0$ fixed, there exists a (large) positive constant $s_0$
such that 
\begin{equation}
\label{ungapasympt}
\begin{split}
\log \mathbb{P}_{U(N)}\bigl(\theta_1>2s\bigr) &= N^2\log \cos (s/2) -\frac{1}{4} \log 
\left(N \sin(s/2)\right) + c_0 + O\left(1/(N\sin (s/2))\right) \\
\frac{d}{ds} \log \mathbb{P}_{U(N)}\bigl(\theta_1>2s\bigr) & = -\frac{N^2}{2} \tan(s/2) -
\frac{1}{8}\cot(s/2) + O(1/(N \sin^2 (s/2))),\nonumber
\end{split}
\end{equation}
for all $n>s_0$ and $2s_0/N \le s \le \pi-\delta$, where $c_0$ is an explicit
constant.
\end{lemma}
\begin{proof}
Clearly, $\mathbb{P}_{U(N)}\bigl(\theta_1>2s\bigr) = \int_{U(N)} 1_{\theta_1 > 2s} \,d\mathbb{P}_{U(N)}$. 
By the rotational invariance of $\mathbb{P}_{U(N)}$, 
this is the same as 
$\int_{U(N)} 1_{\theta_1,\ldots,\theta_N\not\in[0,s]\cup
[2\pi-s,2\pi]} \,d\mathbb{P}_{U(N)}$. Further, by
\cite[Lemma 2]{conrey} and the Weyl integration formula on $U(N)$, this
is the Toeplitz determinant 
$\det_{N\times N}(\int_s^{2\pi-s} e^{i(j-k)\theta}\,d\theta/(2\pi))$,
for which the relevant asymptotics are supplied by 
formulas (8) and (12) in \cite{di}.
\end{proof}
\begin{remark}
Lemma 6.8.3 in \cite{ks} furnishes the 
following interesting factorization of the gap probabilities:
$\mathbb{P}_{U(2N+1)}\bigl(\theta_1>2s\bigr) =
\mathbb{P}_{SO(2N+2)}\bigl(\theta_1>s\bigr) \, 
\mathbb{P}_{USp(2N)}\bigl(\theta_1>s\bigr).$
So, as a direct consequence of Lemmas \ref{uspgapbounds} and
\ref{diLemma}, we obtain
$\log \mathbb{P}_{SO(2N)}\bigl(\theta_1>s\bigr) =
(2+o(1))N^2\log\cos(s/2)$,
provided that $Ns\to\infty$ as $N\to\infty$.
Note that the machinery of orthogonal 
polynomials supplies general, but involved, methods to derive precise 
asymptotics for determinant expressions of gap probabilities
like those in \eqref{ungapasympt}; e.g.\ see
\cite{basor-ehrhardt} and \cite{di2}. 
In the case of $USp(2N)$, for example, the gap probability can be expressed 
as Toeplitz$+$Hankel determinant, or 
by appealing to \eqref{gapprobintegral}, 
as a Hankel determinant. 
\end{remark}
To prove the proposition, first note\footnote{See Lemma 3.1 in \cite{arous-bourgade}, but note it is
missing a factor of $2\pi$.}
\begin{equation}
\label{ungapprobderivative}
\frac{N}{2\pi} \mathbb{P}_{U(N)}(\theta_2-\theta_1 > u) = \left. -\frac{1}{2}
\frac{d}{ds}\mathbb{P}_{U(N)}\bigl(\theta_1>2s\bigr)\right|_{s=u/2}.
\end{equation}
Therefore,
\begin{equation}
\label{unnearestneighbour}
\mathbb{P}_{U(N)} (\theta_2-\theta_1 > u) = \left.-\frac{\pi}{N}\left(\frac{d}{ds} 
\log \mathbb{P}_{U(N)}\bigl(\theta_1>2s\bigr) \right)\right|_{s = u/2}
\mathbb{P}_{U(N)}\bigl(\theta_1>u\bigr).
\end{equation}
It follows from Lemma~\ref{diLemma} that as $N\to \infty$, and 
uniformly for $N^{\nu_1-1} \le u \le 2\pi -\delta$ (for any small constant
$\nu_1>0$ we wish), that the term controlling the behavior in
\eqref{unnearestneighbour} 
is $\mathbb{P}_{U(N)}\bigl(\theta_1>u\bigr)$. Explicitly,
$$
\mathbb{P}_{U(N)}(\theta_2-\theta_1 > u) = \exp((1+o(1)) N^2 \log \cos (u/4))\,.
$$
If we further require $u \le N^{-\nu_2}$ (for any small constant $\nu_2>0$ we
wish),
then $u\to 0$ as $N\to\infty$ and $\log \cos (u/4) \sim -u^2/32$. Therefore, 
\begin{equation}
\label{ungapprobtail}
\mathbb{P}_{U(N)}(\tfrac12[\theta_2-\theta_1] > u) = \exp (- (1+o(1)) N^2 u^2/8 )
\end{equation}
uniformly for $N^{\nu_1-1} \le 2u \le N^{-\nu_2}$. 
Thus, by a similar calculation to \eqref{uspasymptderivation}, 
we obtain the result for $\max_{1\le j\le
M_{\beta}(N)}\frac12[\theta_2(m)-\theta_1(m)]$. 
If, in addition,  one maximizes over all the spacings in each matrix,
thus replacing $\frac12[\theta_2(m)-\theta_1(m)]$ 
by $\max_{1\le j \le N}\frac12[\theta_{j+1}(m)-\theta_j(m)]$,  
then the same result holds.  For clearly 
$$
\mathbb{P}_{U(N)} \bigl(u^-_{\varepsilon,\beta}(N)< 
\max_{1\le m\le M_{\beta}(N)}\tfrac12[\theta_2-\theta_1]\bigr) 
\le
\mathbb{P}_{U(N)}\bigl(u^-_{\varepsilon,\beta}(N)
< \max_{1\le m\le M_{\beta}(N)} \max_{1\le j\le N} 
\tfrac12[\theta_{j+1}(m)-\theta_j(m)] \bigr).
$$
Hence, if the left-hand tends to $1$ then the right-hand side does as
well.  In the opposite direction, we have
$$
N\,\mathbb{P}_{U(N)}\bigl(\tfrac12[\theta_2-\theta_1]
>u^+_{\varepsilon,\beta}(N)\bigr) =o(1),
$$
as $N\to\infty$, by virtue of \eqref{ungapprobtail}. Thus,
\begin{align*}
\log \mathbb{P}_{U(N)}\bigl(\max_{1\le m\le M_{\beta}(N)} &\max_{1\le j\le N} 
\tfrac12[\theta_{j+1}(m)-\theta_j(m)] \le u^+_{\varepsilon,\beta}(N) \bigr) \\
&=M\log \mathbb{P}_{U(N)}\bigl(\max_{1\le j\le N}
\tfrac12[\theta_{j+1}(m)-\theta_j(m)] 
\le u^+_{\varepsilon,\beta}(N)\bigr)  \\ 
&\ge M\log
\bigl(1-N\mathbb{P}_{U(N)}\bigl(\tfrac12[\theta_2-\theta_1]
>u^+_{\varepsilon,\beta}(N)\bigr)\bigr) \\
& =-M N \,\mathbb{P}_{U(N)}\bigl(\tfrac12[\theta_2-\theta_1]
>u^+_{\varepsilon,\beta}(N)\bigr) (1+o(1)),
\end{align*}
where we used the rotational invariance of $\mathbb{P}_{U(N)}$
in the third line. 
Hence maximizing over $1\le j\le N$ does not change our previous calculation
more than does increasing the
number of samples by a factor of $N$, which affects lower order terms only.
Thus, as before, $-M N
\,\mathbb{P}_{U(N)}\bigl(\frac12[\theta_2-\theta_1]>u^+_{\varepsilon,\beta}(N) \bigr)\to0$,
and so $\mathbb{P}_{U(N)}\bigl(\max_{1\le m\le M_{\beta}(N)} \max_{1\le j\le N} 
\frac12[\theta_{j+1}(m)-\theta_j(m)] \le u^+_{\varepsilon,\beta}(N) \bigr) \to 1$.
\qed

\subsection{Proof of Prop.~\ref{converseprop}}
By the growth estimate for the divisor $m$, there is a Hadamard
product $F(z)$ which is entire of finite order, even, satisfies
$\ord_{z=\gamma}F(z)=m(\gamma)$ for all $\gamma\in S$ and does not vanish
outside of $S$.  Moreover, $\frac{F'}{F}(z)$ has at most polynomial growth
in horizontal strips outside of $S$, so that $\frac{F'}{F}(z)h(z)$ decays
rapidly in such strips for any $h$ as in the statement of the proposition.
By the argument principle, for any $c>1/2$ we have
\begin{align*}
\sum_{\gamma\in S}m(\gamma)h(\gamma)&= 
\frac1{2\pi i}\int_{\Im(z)=-c}\frac{F'}{F}(z)h(z)\,dz
-\frac1{2\pi i}\int_{\Im(z)=c}\frac{F'}{F}(z)h(z)\,dz\\
&=\frac1{\pi i}\int_{\Im(z)=-c}\frac{F'}{F}(z)h(z)\,dz.
\end{align*}

Next, let $a\in\{0,1\}$ be such that $(-1)^a=\sgn{d}$, and define
$$\Lambda(s)=|d|^{s/2}\Gamma_{\R}(s+a)
\exp\!\left(\sum_{n=2}^{\infty}
\frac{c_n}{\log{n}}n^{\frac12-s}\right)
\quad\text{and}\quad
\Phi(z)=\Lambda(\frac12+iz).
$$
By the estimate for $c_n$, $\Phi$ is
analytic for $\Im(z)<-1$, where it satisfies
$$
-i\frac{\Phi'}{\Phi}(z)=
\frac12\log|d|
+\frac{\Gamma_\R'}{\Gamma_\R}\!\left(\frac12+a+iz\right)
-\sum_{n=2}^{\infty}c_nn^{-iz}.
$$
Thus, for any $c>1$ we have
\begin{align*}
\frac1{\pi i}&\int_{\Im(z)=-c}\frac{\Phi'}{\Phi}(z)h(z)\,dz\\
&=g(0)\log{|d|}+
\frac1{\pi}\int_{-\infty}^{\infty}
\frac{\Gamma_\R'}{\Gamma_\R}\!\left(\frac12+a+it\right)h(t)\,dt
-2\sum_{n=2}^{\infty}c_ng(\log{n})\\
&=\sum_{\gamma\in S}m(\gamma)h(\gamma)
=\frac1{\pi i}\int_{\Im(z)=-c}\frac{F'}{F}(z)h(z)\,dz.
\end{align*}

Let us now set $f(z)=\frac{F'}{F}(z)-\frac{\Phi'}{\Phi}(z)$ for
$\Im(z)<-1$.  By the above, we see that
$\frac1{\pi}\int_{\Im(z)=-c}f(z)h(z)\,dz=0$ for every
$c>1$ and every suitable choice of test function $h$.
Fix one choice of $h$ and consider the Fourier transform
$$
u(x)=\frac1{2\pi}\int_{\Im(z)=-c}f(z)h(z)e^{-ixz}\,dz.
$$
Note that since $f(z)h(z)$ is holomorphic for $\Im(z)<-1$ and of
rapid decay in horizontal strips, $u(x)$ does not depend on $c$.
Further, for any fixed $x\in\R$, $h(z)\cos(xz)$
is also a suitable test function, so we have
$$u(x)+u(-x)=\frac1{\pi}\int_{\Im(z)=-c}f(z)h(z)\cos(xz)\,dz=0,$$
i.e.\ $u$ is an odd function of $x$.
Combining this with the trivial estimate $u(x)\ll_{c,h}e^{-cx}$,
we get $u(x)\ll_{c,h}e^{-c|x|}$.

Using this estimate for some $c>1$ together with the Fourier
inversion formula
$$f(z)h(z)=\int_{-\infty}^{\infty}u(x)e^{ixz}\,dx,$$
we see that $f(z)h(z)$ continues to an entire function and is odd.
Since $h$ is arbitrary, it follows from a suitable approximation
argument that $f$ continues to an odd
entire function with at most polynomial growth in horizontal strips.
Recalling the definition of $f$ and integrating, we see that $\Phi(z)$
continues to an entire function of finite order satisfying
$\Phi(z)=\epsilon\Phi(-z)$ for some $\epsilon\in\{\pm1\}$,
$\ord_{z=\gamma}\Phi(z)=m(\gamma)$ for every $\gamma\in S$, and
$\Phi(z)\ne0$ for $z\not\in S$.

The remaining statements now essentially follow from the
converse theorem for degree 1 elements of the Selberg class
\cite{kp}, except that the proof given there assumes that
the Dirichlet series $L(s)=\sum_{n=1}^{\infty}a_nn^{-s}$
defined by
$L(s)=\exp\!\left(\sum_{n=2}^{\infty}\frac{c_n}{\log{n}}n^{\frac12-s}\right)$
converges absolutely
for $\Re(s)>1$, which we only know to be true for $\Re(s)>\frac32$.  That
assumption is not necessary, however, and for the sake of completeness
we sketch a simplified proof following the method of \cite{sound}.

First note that the symmetry of $\Phi$ is equivalent to
the functional equation $\Lambda(s)=\epsilon\Lambda(1-s)$.  Next, for any
$\alpha, y>0$ we have
\begin{align*}
2\sum_{n=1}^{\infty}a_n e(n\alpha)e^{-2\pi ny}
&=\frac1{2\pi i}\int_{\Re(s)=2}L(s)\Gamma_\C(s)
(y-i\alpha)^{-s}\,ds\\
&=\frac1{2\pi i}\int_{\Re(s)=2}\Lambda(s)\Gamma_\R(s+1-a)
\bigl[\sqrt{|d|}(y-i\alpha)\bigr]^{-s}\,ds,
\end{align*}
where for any $z$ with positive real part we define
$z^{-s}=\exp(-s\log{z})$ using the principal branch of the logarithm.

By the Phragm\'en--Lindel\"of theorem, the integrand decays rapidly in
vertical strips, so we may shift the contour to $\Re(s)=-3/4$ and apply
the functional equation to obtain
\begin{align*}
2\sum_{n=1}^{\infty}a_n&e(n\alpha)e^{-2\pi ny}
-(1-(-1)^a)\Lambda(0)\\
&=\frac1{2\pi i}\int_{\Re(s)=-3/4}\Lambda(s)\Gamma_\R(s+1-a)
\bigl[\sqrt{|d|}(y-i\alpha)\bigr]^{-s}\,ds\\
&=\frac{\epsilon}{2\pi i}\int_{\Re(s)=7/4}
\Lambda(s)\Gamma_\R(2-a-s)
\bigl[\sqrt{|d|}(y-i\alpha)\bigr]^{s-1}\,ds.
\end{align*}
Expanding $\Lambda(s)$ as
$|d|^{s/2}\Gamma_\R(s+a)\sum_{n=1}^{\infty}a_nn^{-s}$ and using
the identity $\Gamma_\R(s)\Gamma_\R(2-s)=\csc(\pi{s}/2)$, we get
\begin{align*}
2\sum_{n=1}^{\infty}a_n&e(n\alpha)e^{-2\pi ny}
-(1-(-1)^a)\Lambda(0)\\
&=\epsilon\sqrt{|d|}\sum_{n=1}^{\infty}
\frac{a_n}{2\pi i}\int_{\Re(s)=7/4}
n^{-s}\csc\left(\frac{\pi(s+a)}2\right)
\bigl[|d|(y-i\alpha)\bigr]^{s-1}\,ds\\
&=\frac2{\pi}\epsilon i^{a+1}\sqrt{|d|}
\sum_{n=1}^{\infty}\frac{a_n}{n}
\frac{\left(\frac{|d|(\alpha+iy)}{n}\right)^a}
{\frac{|d|(\alpha+iy)}{n}-\frac{n}{|d|(\alpha+iy)}}.
\end{align*}

If $\alpha|d|$ is not an integer then the last line is $O_{\alpha}(1)$
uniformly for $y\in(0,1)$, while if $\alpha|d|=n$ is an integer then
we get $\frac{\epsilon i^aa_n}{\pi\sqrt{|d|}y}+O_\alpha(1)$.
Since the left-hand side is periodic in $\alpha$, we conclude that
$d$ is an integer and $a_n=a_{n+|d|}$, i.e.\ the coefficients $a_n$ are
periodic. Moreover, since $\Lambda(s)$ does not vanish for $\Re(s)>1$,
it follows from \cite[Thm.~4]{sw} that there is a positive integer
$q$ dividing $d$ and a primitive Dirichlet character $\chi\pmod{q}$ such
that $L(s)=D(s)L(s,\chi)$, where
$D(s)=\sum_{n\big|\frac{|d|}{q}}b_nn^{-s}$
for certain coefficients $b_n$, with $b_1=1$.

Let $\Lambda(s,\chi)=q^{s/2}\Gamma_\R(s+a')L(s,\chi)$
be the associated complete $L$-function.  Then we have
\begin{equation}
\label{lambdaratio}
\frac{\Lambda(s)}{\Lambda(s,\chi)}=\left(\frac{|d|}{q}\right)^{s/2}
\frac{\Gamma_\R(s+a)}{\Gamma_\R(s+a')}D(s).
\end{equation}
Moreover, it is easy to see that $D(s)\Lambda(s,\chi)$ does not vanish
in some left half plane. Thus, to avoid concluding from \eqref{lambdaratio}
that $\Lambda(s)$ has poles at negative integers, it must be the case that
and $a'=a$, so that $\chi(-1)=\sgn{d}$.  From this and the functional
equations for $\Lambda(s)$ and $\Lambda(s,\chi)$, it follows that
$\frac{\Lambda(s,\chi)}{\Lambda(s,\overline{\chi})}\frac{D(s)}{D(1-s)}$
is an entire function. Note that for large $T>0$, $D(s)/D(1-s)$ has $O(T)$
zeros and poles with imaginary part in $[-T,T]$.  On the other hand, work
of Fujii \cite{fujii} shows that if $\chi_1$ and $\chi_2$ are distinct,
primitive Dirichlet characters then $\Lambda(s,\chi_1)/\Lambda(s,\chi_2)$
has $\gg T\log{T}$ zeros and poles in that region. Thus, we must have
$\chi=\overline{\chi}$, i.e.\ $\chi$ is quadratic and
$q\sgn{d}$ is a fundamental discriminant.

Therefore, by \eqref{lambdaratio} and the functional equations for
$\Lambda(s)$ and $\Lambda(s,\chi)$, $D(s)$ satisfies the functional equation
\begin{equation}
\label{Dfunceq}
D(s)=\epsilon\left(\frac{|d|}{q}\right)^{\frac12-s}D(1-s).
\end{equation}
Next, from the formula for $L(s)$, we have
$$
\frac{D'}{D}(s)=\sum_{\substack{n\ge2\\n\bigl|\bigl(\frac{|d|}q\bigr)^\infty}}
\left(\frac{\Lambda(n)\chi(n)}{\sqrt{n}}-c_n\right)
n^{\frac12-s},
$$
where the notation $n\bigl|\bigl(\frac{|d|}q\bigr)^\infty$ means that
$n$ is composed only of primes dividing $|d|/q$.

Now, from \eqref{Dfunceq} and the estimate
$\frac{\Lambda(n)\chi(n)}{\sqrt{n}}-c_n=O(n^{-\delta})$ it follows that
$\frac{D'}{D}(s)$ is entire, and thus $D(s)=1$ identically.
Finally, invoking \eqref{Dfunceq} once more,
we have $|d|=q$ and $\epsilon=1$.
\qed

\bibliographystyle{amsplain}
\bibliography{paper}

\providecommand{\bysame}{\leavevmode\hbox to3em{\hrulefill}\thinspace}
\providecommand{\MR}{\relax\ifhmode\unskip\space\fi MR }
\providecommand{\MRhref}[2]{%
  \href{http://www.ams.org/mathscinet-getitem?mr=#1}{#2}
}
\providecommand{\href}[2]{#2}
\begin{thebibliography}{10}

\bibitem{ajtai}
M.~Ajtai, \emph{The shortest vector problem in {$L_2$} is {$NP$}-hard for
  randomized reductions (extended abstract)}, Proceeding STOC '98 Proceedings
  of the thirtieth annual ACM symposium on Theory of computing, 1998.

\bibitem{arous-bourgade}
G\'erard~Ben Arous and Paul Bourgade, \emph{Extreme gaps between eigenvalues of
  random matrices}, arXiv:1010.1294, 2011.

\bibitem{basor-ehrhardt}
Estelle~L. Basor and Torsten Ehrhardt, \emph{Asymptotic formulas for
  determinants of a sum of finite {T}oeplitz and {H}ankel matrices}, Math.
  Nachr. \textbf{228} (2001), 5--45. \MR{1845906 (2002d:47041)}

\bibitem{conrey}
Brian Conrey, \emph{Notes on eigenvalue distributions for the classical compact
  groups}, Recent perspectives in random matrix theory and number theory,
  London Math. Soc. Lecture Note Ser., vol. 322, Cambridge Univ. Press,
  Cambridge, 2005, pp.~111--145. \MR{2166460 (2006g:11177)}

\bibitem{costa-harvey}
Edgar Costa and David Harvey, \emph{Faster deterministic integer
  factorization}, arXiv:1201.2116, 2011.

\bibitem{davenport}
Harold Davenport, \emph{Multiplicative number theory}, third ed., Graduate
  Texts in Mathematics, vol.~74, Springer-Verlag, New York, 2000, Revised and
  with a preface by Hugh L. Montgomery. \MR{1790423 (2001f:11001)}

\bibitem{di}
P.~Deift, A.~Its, I.~Krasovsky, and X.~Zhou, \emph{The {W}idom-{D}yson constant
  for the gap probability in random matrix theory}, J. Comput. Appl. Math.
  \textbf{202} (2007), no.~1, 26--47. \MR{2301810 (2008e:82027)}

\bibitem{di2}
Percy Deift, Alexander Its, and Igor Krasovsky, \emph{Asymptotics of
  {T}oeplitz, {H}ankel, and {T}oeplitz+{H}ankel determinants with
  {F}isher-{H}artwig singularities}, Ann. of Math. (2) \textbf{174} (2011),
  no.~2, 1243--1299. \MR{2831118 (2012h:47063)}

\bibitem{fujii}
Akio Fujii, \emph{On the zeros of {D}irichlet {$L$}-functions. {V}}, Acta
  Arith. \textbf{28} (1975/76), no.~4, 395--403. \MR{411182 (81g:10057a)}

\bibitem{gmp-ecm}
\emph{{GMP-ECM}}, \url{http://ecm.gforge.inria.fr/}.

\bibitem{glpk}
\emph{{GNU Linear Programming Kit}}, \url{http://www.gnu.org/software/glpk/}.

\bibitem{kp}
Jerzy Kaczorowski and Alberto Perelli, \emph{On the structure of the {S}elberg
  class. {I}. {$0\leq d\leq 1$}}, Acta Math. \textbf{182} (1999), no.~2,
  207--241. \MR{1710182 (2000h:11097)}

\bibitem{ks}
Nicholas~M. Katz and Peter Sarnak, \emph{Random matrices, {F}robenius
  eigenvalues, and monodromy}, American Mathematical Society Colloquium
  Publications, vol.~45, American Mathematical Society, Providence, RI, 1999.
  \MR{1659828 (2000b:11070)}

\bibitem{katz-sarnak}
\bysame, \emph{Zeroes of zeta functions and symmetry}, Bull. Amer. Math. Soc.
  (N.S.) \textbf{36} (1999), no.~1, 1--26. \MR{1640151 (2000f:11114)}

\bibitem{krasikov}
I.~Krasikov, \emph{Uniform bounds for {B}essel functions}, J. Appl. Anal.
  \textbf{12} (2006), no.~1, 83--91. \MR{2243854 (2008c:33002)}

\bibitem{krasovsky}
I.~V. Krasovsky, \emph{Gap probability in the spectrum of random matrices and
  asymptotics of polynomials orthogonal on an arc of the unit circle}, Int.
  Math. Res. Not. (2004), no.~25, 1249--1272. \MR{2047176 (2005d:60086)}

\bibitem{LLS}
Youness Lamzouri, Xiannan Li, and Kannan Soundararajan, \emph{The least
  quadratic non-residue, values of {$L$}-functions at $s=1$, and related
  problems}, {\tt arXiv:1309.3595} (2013).

\bibitem{lll}
A.~K. Lenstra, H.~W. Lenstra, Jr., and L.~Lov{\'a}sz, \emph{Factoring
  polynomials with rational coefficients}, Math. Ann. \textbf{261} (1982),
  no.~4, 515--534. \MR{682664 (84a:12002)}

\bibitem{martin}
Robert Martin, \emph{Bandlimited functions, curved manifolds, and self-adjoint
  extensions of symmetric operators}, ProQuest LLC, Ann Arbor, MI, 2008, Thesis
  (Ph.D.)--University of Waterloo (Canada). \MR{2712567}

\bibitem{mehta}
Madan~Lal Mehta, \emph{Random matrices}, third ed., Pure and Applied
  Mathematics (Amsterdam), vol. 142, Elsevier/Academic Press, Amsterdam, 2004.
  \MR{2129906 (2006b:82001)}

\bibitem{montgomery-odlyzko}
Hugh~L. Montgomery and Andrew~M. Odlyzko, \emph{Large deviations of sums of
  independent random variables}, Acta Arith. \textbf{49} (1988), no.~4,
  427--434. \MR{937937 (89m:11075)}

\bibitem{odlyzko}
A.~M. Odlyzko, \emph{On the distribution of spacings between zeros of the zeta
  function}, Math. Comp. \textbf{48} (1987), no.~177, 273--308. \MR{866115
  (88d:11082)}

\bibitem{odlyzko-teriele}
A.~M. Odlyzko and H.~J.~J. te~Riele, \emph{Disproof of the {M}ertens
  conjecture}, J. Reine Angew. Math. \textbf{357} (1985), 138--160. \MR{783538
  (86m:11070)}

\bibitem{omar}
Sami Omar, \emph{Non-vanishing of {D}irichlet {$L$}-functions at the central
  point}, Algorithmic number theory, Lecture Notes in Comput. Sci., vol. 5011,
  Springer, Berlin, 2008, pp.~443--453. \MR{2467864 (2009k:11133)}

\bibitem{pollard}
J.~M. Pollard, \emph{Theorems on factorization and primality testing}, Proc.
  Cambridge Philos. Soc. \textbf{76} (1974), 521--528. \MR{0354514 (50 \#6992)}

\bibitem{rubinstein}
Michael Rubinstein, \emph{Low-lying zeros of {$L$}-functions and random matrix
  theory}, Duke Math. J. \textbf{109} (2001), no.~1, 147--181. \MR{1844208
  (2002f:11114)}

\bibitem{rubinstein-sarnak}
Michael Rubinstein and Peter Sarnak, \emph{Chebyshev's bias}, Experiment. Math.
  \textbf{3} (1994), no.~3, 173--197. \MR{1329368 (96d:11099)}

\bibitem{sw}
Eric Saias and Andreas Weingartner, \emph{Zeros of {D}irichlet series with
  periodic coefficients}, Acta Arith. \textbf{140} (2009), no.~4, 335--344.
  \MR{2570109 (2010m:11107)}

\bibitem{sound}
K.~Soundararajan, \emph{Degree 1 elements of the {S}elberg class}, Expo. Math.
  \textbf{23} (2005), no.~1, 65--70. \MR{2133337 (2006c:11104)}

\bibitem{strassen}
Volker Strassen, \emph{Einige {R}esultate \"uber {B}erechnungskomplexit\"at},
  Jber. Deutsch. Math.-Verein. \textbf{78} (1976/77), no.~1, 1--8. \MR{0438807
  (55 \#11713)}

\bibitem{swamy}
M.~N.~S. Swamy, \emph{Further properties of {M}organ-{V}oyce polynomials},
  Fibonacci Quart. \textbf{6} (1968), no.~2, 167--175. \MR{0237470 (38 \#5752)}

\end{thebibliography}
\end{document}